\documentclass[12pt]{article}
\usepackage{etex}
\usepackage{amsmath,amsfonts,amsthm,amssymb,thmtools,mathtools}
\usepackage{setspace}
\usepackage{enumerate}
\usepackage{fancyhdr}
\usepackage{lastpage}
\usepackage{extramarks}
\usepackage{chngpage}
\usepackage{soul}
\usepackage[usenames,dvipsnames]{color}
\usepackage{graphicx,float,wrapfig}
\usepackage{ifthen}
\usepackage{listings}
\usepackage{courier}
\usepackage{url}
\usepackage[all,cmtip,2cell]{xy}
\UseTwocells
\usepackage{pifont}
\usepackage{longtable}
\usepackage[titletoc,toc,title]{appendix}

\newcommand{\cat}[1]{\mathbf{#1}}
\newcommand{\obj}{\mathrm{Obj}}
\newcommand{\id}{\mathrm{id}}
\newcommand{\ver}{\mathrm{Ver}}
\newcommand{\hor}{\mathrm{Hor}}
\newcommand{\dbl}{\mathrm{Dbl}}
\newcommand{\dom}{\mathrm{dom}}
\newcommand{\hdom}{\mathrm{hdom}}
\newcommand{\vdom}{\mathrm{vdom}}
\newcommand{\cod}{\mathrm{cod}}
\newcommand{\hcod}{\mathrm{hcod}}
\newcommand{\vcod}{\mathrm{vcod}}
\newcommand{\vmeet}{\wedge_v}
\newcommand{\hmeet}{\wedge_h}
\newcommand{\im}{\mathrm{Im}}
\newcommand{\cis}{\mathit{IS}}
\newcommand{\cig}{\mathit{IG}}
\newcommand{\cdig}{\mathit{DIG}}
\newcommand{\cdis}{\mathit{DIS}}

\newtheorem{lemma}{Lemma}
\newtheorem{proposition}{Proposition}
\newtheorem*{notation}{Notation}
\newtheorem{corollary}{Corollary}
\newtheorem{theorem}{Theorem}
\numberwithin{theorem}{section}
\numberwithin{lemma}{section}
\numberwithin{proposition}{section}

\numberwithin{conjecture}{section}

\theoremstyle{definition}
\newtheorem{definition}{Definition}[section]
\newtheorem{example}[definition]{Example}

\theoremstyle{definition} \newtheorem{note}{Note}
\theoremstyle{definition} \newtheorem{construction}[theorem]{Construction}

\begin{document}
\title{On Double Inverse Semigroups}
\author{Darien DeWolf and Dorette Pronk}
\date{}

\maketitle

\begin{abstract}
We show that the two binary operations in double inverse semigroups, as considered by Kock [2007], necessarily coincide.
\end{abstract}
%
%
\section{Introduction}
%
%
There has recently been much interest in higher order group-like algebraic structures. An example of such higher dimensional structures are $n$-fold groupoids. The category $\cat{Grpd}^n$ of  \emph{$n$-fold groupoids} is defined inductively as follows: the category of 1-fold groupoids is $\cat{Grpd},$ the usual category of (small) groupoids, the category of  2-fold groupoids is the category of double groupoids (groupoids internal to $\cat{Grpd})$ and the category of $n$-fold groupoids is the category of groupoids internal to $\cat{Grpd}^{n-1}.$ $n$-fold groupoids are important for homotopy theory because they are conjectured to form an algebraic model for homotopy $n$-types. Recall that a topological space $X$ is said to be a \emph{homotopy $n$-type} whenever its $k$-homotopy groups $\pi_k(X)$ are trivial for all $k > n.$ The homotopy hypothesis conjectures that the category of homotopy $n$-types is equivalent to the category of $n$-fold groupoids (The single-object case of this is handled by Loday \cite{LODAY1982179} using $n$-cat-groups). In \cite{brownwebgroup}, Brown introduces two-dimensional group theory, defining a double group as a single-object 2-fold groupoid.

A more element-based description of a $2$-fold groupoid is obtained by defining a \emph{double group} $(G,\odot, \circledcirc)$ to be a set equipped with two group operations that satisfy the middle-four interchange law. That is, for all $a,b,c,d\in G,$ the following holds:
\[ (a\odot b)\circledcirc (c\odot d) = (a\circledcirc c)\odot ( b\circledcirc d).\]
The Eckmann-Hilton argument \cite{eckmannhilton61}, however, implies that such a structure exists if and only if both group operations are commutative and coincide. The crux of the Eckmann-Hilton argument happens to lie in the monoid structure of groups (in particular, the existence of identity elements for the two associative operations); the interchange law implies that both operations are commutative, share units and, in fact, coincide. 

On the other extreme, Edmunds \cite{edmunds2013} has recently studied what are known as \emph{double magma}: sets equipped with two binary operations that satisfy the interchange law  (but do not necessarily have units and are not necessarily associative). Edmunds explores the structure of such objects and provides constructions of double magma with two different binary operations.

Kock \cite{kock07} defines the notion of \emph{double semigroup}: a set equipped with two associative binary operations that satisfy the middle-four interchange. Without units, the Eckmann-Hilton argument does not apply and there exist double semigroups in which the operations are neither commutative nor coincide. Following Edmunds' terminology, we call a double semigroup in which the two operations do not coincide \emph{proper} \cite{edmunds2013}. Otherwise, we call it \emph{improper}. A \emph{double cancellative semigroup} is a double semigroup whose two operations are both left and right cancellative. Kock uses a tile-sliding argument to show that double cancellative semigroups are neccesarily commutative. Kock's argument can be extended to show that double cancellative semigroups are also improper \cite[Proposition 3.2.4]{dewolf13}. A semigroup $(S,\odot)$ is an inverse semigroup whenever, for any element $a\in S,$ there is a unique element $a^\odot\in S$ such that $a\odot a^\odot \odot a  = a$ and $a^\odot \odot a \odot a^\odot = a^\odot.$ A \emph{double inverse semigroup} is a double semigroup whose two operations are both inverse semigroup operations.  Kock uses another tile-sliding argument to show that double inverse semigroups are necessarily commutative. His argument, however, can not be immediately extended to show that double inverse semigroups are improper. Therefore, we seek another method to determine whether or not proper double inverse semigroups exist.

Studying pseudogroups of transformations in differential geometry, Ehresmann \cite{ehresmann57} defined an inductive groupoid as a groupoid with a  partial order on its arrows satisfying conditions on how the partial order behaves with respect to composition, inverses and (co)restrictions. In Ehresmann's work, there is a link between inductive groupoids and inverse semigroups. This is made more precise by Lawson \cite{lawson98}, who provides explicit constructions of an inverse semigroup given an inductive groupoid, and vice versa, and shows that these constructions induce an isomorphism between the category of inverse semigroups and the category of inductive groupoids, a result known as the Ehresmann-Schein-Nambooripad Theorem. 

This paper introduces the notion of a \emph{double inductive groupoid}, a double groupoid with a partial order on the double cells which shares the properties of the partial order in an inductive groupoid. We extend the isomorphism between inverse semigroups and inductive groupoids provided by Lawson to one between double inverse semigroups and double inductive groupoids.

This isomorphism is used, then, to develop some intuition about the structure of double inverse semigroups. It follows immediately that the two inverse semigroup operations must coincide on the shared idempotents, thus collapsing the inductive groupoid structures on the shared idempotents. It is then shown that double inverse semigroups induce natural single-object double inductive groupoids indexed by the shared idempotents, which are shown to be groups. It follows, then, that double inductive groupoids (and thus double inverse semigroups) are presheaves of Abelian groups on meet-semilattices and, finally, that double inverse semigroups must be improper.

Section 2 of this paper discusses double semigroups and introduces the problem of finding proper double inverse semigroups. Section 3 discusses inductive groupoids and constructions relating them to inverse semigroups. Section 4 introduces double inductive groupoids and the constructions relating them to double inverse semigroups. Finally, Section 5 establishes that double inductive groupoids, and thus double inverse semigroups, are exactly presheaves of Abelian groups on meet-semilattices.

\begin{notation}\emph{
\begin{itemize}
\item We identify the objects of a category with its identity morphism via the correspondence $ a \leftrightarrow 1_a.$
\item If $f:A\rightarrow B$ is an arrow in some category, we denote its domain as $f \dom = A$ and its codomain as $f \cod = B.$
\item If $f$ and $g$ are arrows in a category with $f\cod = g\dom,$ we denote their composite as $f; g,$ or simply $fg.$ We note the use of postfix notation for composition and evaluation.
\end{itemize}}
\end{notation}
%
%
\section{Double Semigroups}
\label{doublesemigroups}
%
%
It is an immediate consequence of the Eckmann-Hilton argument \cite{eckmannhilton61} that if a set is equipped with two group (in fact, monoid) operations satisfying the middle-four interchange law, then both operations must be commutative and must coincide. That is,  double groups, when defined in this way, are essentially Abelian groups. The Eckmann-Hilton argument, however, draws all of its power from the existence of shared multiplicative units. In this section, we consider objects in which the operations do not have units. The following definition was introduced by Kock \cite{kock07}.
\begin{definition}\emph{
A \emph{double semigroup} is an ordered triple $(S,\odot, \circledcirc)$ consisting of a set $S$ and two associative binary operations $\odot$ and $\circledcirc$ that satisfy the following middle-four interchange law: for any $a,b,c,d\in S,$ 
\[( a\odot b) \circledcirc (c\odot d) = (a \circledcirc c) \odot (b\circledcirc d). \qedhere \]}
\end{definition}
Using Edmunds' \cite{edmunds2013} terminology, we make the following definition:
\begin{definition}\emph{
A double semigroup $(S,\odot, \circledcirc)$ is said to be \emph{proper} whenever $\odot \neq \circledcirc.$ A double semigroup that is not proper is said to be \emph{improper}.}
\end{definition}
\begin{definition}\emph{
Let $(S,\circledcirc ,\odot )$ and $(S',\circledcirc ',\odot ')$ be double semigroups. A \emph{double semigroup homomorphism} $\varphi:S\rightarrow S'$ is a function $\varphi:S\rightarrow S'$ such that, for all $a,b\in S,$ $(a \odot  b)\varphi = a\varphi \odot'  b\varphi$ and $(a\circledcirc  b)\varphi = a\varphi \circledcirc'  b \varphi.$}
\end{definition}
\begin{example}\emph{
\begin{enumerate}[(i)]
\item Given any commutative semigroup  $(S,\odot),$ there is an improper double semigroup $(S,\odot,\odot).$
\item In particular, there is the trivial double semigroup, $1 = (\{0\}, \odot, \odot).$
\item (Left and right projections) Given any set $S,$ construct a (proper) double semigroup $(S,\odot, \circledcirc)$ by defining 
\[ a\odot b = a, \mbox{ for all } a,b\in S \mbox{ and } a\circledcirc b = b, \mbox{ for all } a,b\in S.\]
We note that these two semigroups are usually called the left- and right-zero semigroups on $S.$ The middle-four interchange law is indeed satisfied, for
\[ (a\odot b)\circledcirc (c\odot d) = c\odot d = c \mbox{ and } (a\circledcirc c)\odot (b\circledcirc d) = a\circledcirc c = c. \qedhere\]
\end{enumerate}}
\end{example}
Recall the following definitions from semigroup theory:
\begin{definition}\emph{
A semigroup $(S,\odot)$ is said to be an \emph{inverse semigroup} whenever $S$ possesses the following property: for any $a\in S,$ there exists a \emph{unique} element $a^\odot \in S$ such that $a\odot a^\odot \odot a = a$ and $a^\odot = a \odot a^\odot \odot a.$ Given an element $a\in S,$ the unique element $a^\odot\in S$ is called the \emph{semigroup inverse of $a$}.}
\end{definition}
\begin{example}\emph{
A \emph{partial automorphism} of a set $S$ is an isomorphism $\varphi:U \stackrel{\sim}{\to} V,$ where $U,V\subseteq S.$ For a given set $S,$ its collection of partial isomorphisms forms an inverse semigroup under (generalised) composition. The relationship between inverse semigroups and partial automorphisms can be thought of as a generalisation of that between symmetries and groups in the sense that every inverse semigroup is isomorphic to a subsemigroup of some symmetric inverse semigroup of partial automorphisms. This is known as the Wagner-Preston Theorem and we refer the reader to \cite{lipscomb96} for a development of this theory.}
\end{example}
\begin{definition}\emph{
Let $(S,\odot)$ be a semigroup. An element $e\in S$ is said to be \emph{idempotent} whenever $e\odot e = e.$ The set of all idempotents of $S$ is denoted by $E(S,\odot).$}
\end{definition}
The following is a characterisation of inverse semigroups. This is Theorem 3 in Chapter 1 of \cite{lawson98} and will be used in Section  \ref{mainresult} to prove our main result.
\begin{theorem}
Let $(S,\odot)$ be a regular semigroup. Then $(S,\odot)$ is an inverse semigroup if and only if $E(S,\odot)$ is commutative. \qed
\end{theorem}
Finally, we may define what is meant by a double inverse semigroup:
\begin{definition}\emph{
A double semigroup $(S,\odot, \circledcirc)$ is said to be a \emph{double inverse semigroup} whenever both $(S,\odot)$ and $(S,\circledcirc)$ are inverse semigroups.}
\end{definition}
\begin{theorem}[Kock \cite{kock07}]
\label{inversecommutethm}
Double inverse semigroups are commutative (in the sense that both of its operations are). \qed
\end{theorem}
%
%
\section{Inductive Groupoids}
\label{inductivegroupoids}
%
%
To facilitate some understanding of the structure of double inverse semigroups, we first explore the structure of inverse semigroups. Lawson \cite{lawson98} has given an explicit and comprehensive exposition of Ehresmann, Schein and Nampooripad's characterization of inverse semigroups using inductive groupoids -- groupoids equipped with a property-rich partial order on its arrows. This section is a review of that work.
\begin{definition}\emph{
Let $(G,\bullet)$ be a groupoid and let $\leq$ be a partial order defined on the arrows of $G.$ We call $(G,\bullet,\leq)$ an \emph{ordered groupoid} whenever the following conditions are satisfied (where $G_1$ is the set of arrows in $G):$
\begin{enumerate}[(i)]
\item For all $x,y\in G_1,$ $x\leq y$ implies $x^{-1} \leq y^{-1}.$
\item For all $x,y,u,v\in G_1,$ if $x\leq y,$ $u \leq v$ and the composites $xu$ and $yv$ exist in $G,$ then $xu \leq yv.$
\item Let $x\in G_1$ and let $e$ be an object in $G$ such that $e \leq x\mathrm{dom}.$ Then there is a unique element $(e {}_* |x)\in G_1,$ called the restriction of $x$ by $e,$ such that $(e {}_* |x) \leq x$ and $(e {}_* |x)\mathrm{dom} = e.$
\item Let $x\in G_1$ and let $e$ be an object in $G$ such that $e \leq x\mathrm{cod}.$ Then there is a unique element $(x| {}_* e)\in G_1,$ called the corestriction of $x$ by $e,$ such that $(x| {}_* e) \leq x$ and $(x| {}_* e)\mathrm{cod} = e.$
\end{enumerate}
We say that $G$ is an \emph{inductive groupoid} if the further condition that the objects of $G$ (or, equivalently by our identification, the identity arrows of $G)$ form a meet-semilattice is satisfied.}
\end{definition}
\begin{definition}\emph{
A functor $F:G\rightarrow G'$ between two inductive groupoids is called \emph{inductive} if it preserves both the order and the meet operation on the set of objects and the order on the set of arrows.}
\end{definition}
Before we continue, we will point out that property (ii) of inductive groupoids, in which inverses preserve the order, indicates that  inductive groupoids are not a generalisation of ordered groups. For example, consider the usual order on the integers as an additive group. In the integers, we have $a \leq b$ implies $-a \geq -b,$ thus reversing the direction of the inequality for inverses. In fact, note that an inductive groupoid with one object has a trivial partial ordering.

It was mentioned in Section \ref{doublesemigroups} that every inverse semigroup is isomorphic to a subsemigroup of a symmetric inverse semigroup of partial isomorphisms. This isomorphism and the description of semigroups of partial automorphisms is the key motivation for the given definition of an inductive groupoid. In particular, the requirement that inverses preserve the order makes sense when we think of partial automorphisms as being ordered by inclusion of both their domains and images. We will demonstrate this now by constructing an inductive groupoid from partial automorphisms of a set.
\begin{example}\emph{
Consider the set $A = \{ a,b\}.$ Consider the following five partial automorphisms of $A:$
\begin{itemize}
\item $\id_A : a \mapsto a, b\mapsto b,$
\item $\sigma: a\mapsto b, b\mapsto a,$
\item $\id_a : a\mapsto a,$
\item $f : a\mapsto b$ 
\item $f^{-1} : b\mapsto a$ and
\item $\id_b: b\mapsto b.$
\end{itemize}
Each of these automorphisms have inverses, of course, and thus the following, call it $G,$ is a groupoid:
\[ 
\xymatrix{ 
\{a,b\} \ar@/^2pc/@(dl,ul)[]^{\id_A} \ar@(dr,ur)[]_\sigma && \{a\} \ar@<-0.5em>[r]_f \ar@(dl,ul)[]^{\id_a} & \{b\} \ar@<-0.5em>[l]_{f^{-1}}  \ar@(dr,ur)[]_{\id_b}
}  
\]
It is easy to check that if a partial order is defined on the arrows by $\alpha \leq \beta$ if and only if $\alpha\dom \subseteq \beta\dom $ and $\alpha \im \subseteq \beta \im,$ then the conditions required for $G$ being called an inductive groupoid are satisfied:
\begin{enumerate}[(i)]
\item Inverses preserve the partial order on the arrows. For example, $f \im = \{ b\} \subseteq \{a,b\} = \sigma \im,$ so that $ f\leq \sigma.$ Also, $ f^{-1}\im = \{a\} \subseteq \{a,b\} = \sigma^{-1}\im,$ so that $ f^{-1} \leq \sigma^{-1}.$
\item The compositions are well-behaved with respect to the partial order on the arrows. For example,  $f \leq \id_b$ and $\id_b \leq \id_b,$ while $f;\id_b = f \leq \id_b = \id_b; \id_b.$
\item We have unique restrictions when needed. Note that the objects of this groupoid are ordered by inclusion. Then both $\{a\} \leq \sigma\dom = \{a,b\}$ and $ \{b\} \leq \sigma\dom.$ The restrictions of $\sigma$ to $\{a\}$ and $\{b\}$ are $( \{a\} {}_*| \sigma) = f$ and $(\{b\}{}_*| \sigma) = f^{-1},$ respectively.
\item Similarly, we have all required corestrictions. \qedhere
\end{enumerate}}
\end{example}
We will now detail Lawson's two constructions (\cite[p.108]{lawson98} and \cite[p.112]{lawson98}): one for constructing inverse semigroups from inductive groupoids, and vice versa:
\begin{construction} \emph{
\label{igtois}
Given an inductive groupoid $(G,\bullet,\leq,\wedge),$ construct an inverse semigroup $\cis(G) = (S,\odot )$ with $S = G_1$ and, for any $a,b\in S,$
\[ a\odot  b = (a| {}_* a\mathrm{cod}\wedge b\mathrm{dom}) \bullet (a\mathrm{cod}\wedge b\mathrm{dom} {}_* |b).\]
Since, for any $a\in S = G_1,$ $a\mathrm{cod}$ and $a\mathrm{dom}$ are objects in $G,$ this product is always defined, since $G_0$ is a meet-semilattice with respect to the $\wedge$ operation used.}
\end{construction}
\begin{theorem}[\cite{lawson98}] \label{inversethm}
For any inductive groupoid $G,$ $\mathit{IS}(G)$ as defined in Construction \ref{igtois} is an inverse semigroup. \qed
\end{theorem}
\begin{construction}\emph{
\label{istoig}
Given an inverse semigroup $(S,\odot)$ with the natural partial ordering $\leq,$ define a groupoid, $\cig (S),$ with the following data:
\begin{enumerate}[--]
\item Its objects are the idempotents of $S;$ ${\cig(S)}_0 = E(S) .$
\item Its arrows are the elements of $S.$ For any $s\in S,$ define $s\mathrm{dom} = s\odot s^\odot $ and $s\mathrm{cod} = s^\odot \odot s,$ where $s^\odot $ is the inverse of $s.$
\begin{itemize}
\item For any $a,b\in \cig(S)_1,$ if $a\mathrm{cod} = b\mathrm{dom},$ define the composite $a\bullet b = a \odot b,$ the product in $S.$ This composition is well defined (i.e., the composite has the proper domain and codomain) and therefore inherently associative: if $a\bullet b$ is defined, then $a\mathrm{cod} = b\mathrm{dom},$ or $a^\odot \odot a = b\odot b^\odot .$ Then $(a\bullet b)\mathrm{dom} =  (a\odot b)\mathrm{dom} = (a\odot b)\odot (a\odot b)^\odot  = a\odot b\odot b^\odot \odot a^\odot  = a\odot a^\odot \odot a\odot a^\odot  = a\odot a^\odot  = a\mathrm{dom}.$ Similarly, we have $(a\bullet b)\mathrm{cod} = b\mathrm{cod}.$ 
\item For any $a\in S,$ $(a\odot a^\odot ) \odot a = a\odot (a^\odot \odot a) = a.$ Then $a\odot a^\odot$ and $a^\odot \odot a$ act as the left and right identity, respectively, for composition (recall that we have identified the objects with identity arrows).
\item It follows, then, that every arrow is an isomorphism with $a^{-1} = a^\odot ,$ since $a^{-1}\bullet a = a^\odot \odot a = a\mathrm{cod}$ and $a\bullet a^{-1} = a\odot a^\odot  = a\mathrm{dom}$ (recall that we have identified the objects with identity arrows). \qedhere
\end{itemize}
\end{enumerate}}
\end{construction}
\begin{theorem}[\cite{lawson98}]
$\cig(S)$ is an inductive groupoid with, for all $a\in \cig(S),$ $(a| {}_* e) = ae$ for all objects $e\leq a\cod$ and $(e {}_* |a) = ea$ for all objects $e \leq a\dom.$ \qed
\end{theorem}
\begin{notation}\emph{
Denote the category of inverse semigroups and semigroup homomorphisms as $\cat{IS}.$ Denote the category of inductive groupoids and inductive functors as $\cat{IG}.$ }
\end{notation}
Being able to construct an inductive groupoid given any inverse semigroup and vice versa is extremely useful. Specifically, Lawson \cite{lawson98} provides proofs that these constructions induce functors $F:\cat{IG} \rightarrow \cat{IS}$ and $F': \cat{IS}\rightarrow \cat{IG}.$  This then establishes the following result relating the categories $\cat{IS}$ and $\cat{IG}:$
\begin{theorem}[Ehresmann-Schein-Nambooripad \protect{\cite[4.1.8]{lawson99}}]
The functors $F:\cat{IG} \rightarrow \cat{IS}$ and $F': \cat{IS}\rightarrow \cat{IG}$ form an isomorphism of categories. \qed
\end{theorem}
We will demonstrate this result by giving an example of an inverse semigroup and then calculating its corresponding inductive groupoid.
\begin{example}\emph{
Consider the following inverse semigroup $(S,\odot)$ (semigroup (5,415) of the \texttt{smallsemi} package of GAP \cite{GAP4}):
\[\begin{tabular}{c|ccccc}
$\odot$ &1 &2 &3 &4 &5  \\
\hline 
1 & 1&1&1&1&1 \\
2 & 1&1&4&1&2 \\
3 & 1&5&1&3&1 \\
4 & 1&2&1&4&1 \\
5 & 1&1&3&1&5
\end{tabular}\]
It is the case that $S$ is the only non-commutative inverse semigroup of order 5 containing a non-idempotent element. We note that idempotents of $S$ are $E(S) = \{ 1,4,5 \}.$ Note that $ 1 = 1\odot 4 = 1 \odot 5,$ so that $1 \leq 4$ and $1\leq 5.$ The elements 4 and 5 are incomparable, so that the semilattice structure on $E(S)$ is 
\[
\xymatrix{
4 \ar@{-}[rd] & & 5 \ar@{-}[ld] \\
& 1
}
\]
It is routine to check that the inverses of the elements in $S$ are 
\[  1^\odot  = 1, 2^\odot  = 3, 3^\odot = 2, 4^\odot  = 4, \mbox{ and } 5^\odot  = 5. \]
We note that $ 3\dom = 3\odot 3^\odot = 3\odot 2 = 5,$ $3\cod = 3^\odot\odot 3 = 2\odot 3 = 4,$ $2\dom = 4,$ and $2\cod = 5.$ The associated inductive groupoid, then, has
\begin{enumerate}[--]
\item Objects: $\{1,4,5\}$
\item (Non-identity) Arrows: $ \{ 2: 4\rightarrow 5, 3:5\rightarrow 4\}$
\end{enumerate}
And has the visual appearance (although this representation does not reflect the order structure and (co)restrictions)
\[ \xymatrix{ \raisebox{.5pt}{\textcircled{\raisebox{-.9pt} {1}}} \ar@(dl,ul)[]^1 && \raisebox{.5pt}{\textcircled{\raisebox{-.9pt} {4}}}\ar@/^0.5pc/[r]^2 \ar@(dl,ul)[]^4 & \raisebox{.5pt}{\textcircled{\raisebox{-.9pt} {5}}} \ar@/^0.5pc/[l]^3 \ar@(dr,ur)[]_5} \]}
\end{example}
%
%
\section{Double Inductive Groupoids}
\label{doubleinductivegroupoids}
%
In this section, we will extend Lawson's constructions of inductive groupoids and inverse semigroups to constructions involving double inverse semigroups. We will then construct an isomorphism of categories analogous to that between $\cat{IG}$ and $\cat{IS}.$  The first step in extending these constructions is to introduce the notion of a double inductive groupoid.
\begin{definition}\emph{
A \emph{double inductive groupoid}
\[ \mathcal{G} = (\obj(\mathcal{G}),\ver(\mathcal{G}),\hor(\mathcal{G}),\dbl(\mathcal{G}), \leq, \lesssim)\]
 is a double groupoid (i.e., a double category in which every vertical and horizontal arrow is an isomorphism and each double cell is an isomorphism with respect to both the horizontal and vertical composition) such that
\begin{enumerate}[(i)]
\item \label{horind} $(\ver(\mathcal{G}),\dbl(\mathcal{G}))$ is an inductive groupoid.
\begin{itemize}
\item  We denote the composition in this inductive groupoid -- the horizontal composition from $\dbl(\mathcal{G})$ -- with $\circ.$  We denote the partial order on this groupoid as $\leq.$ If $e$ and $f$ are horizontal identity cells (vertical arrows), we denote their meet as $e\hmeet f.$ For a cell $\alpha\in \dbl(\mathcal{G})$ and a vertical arrow $e \in \ver(\mathcal{G})$ such that $e\leq \alpha \hdom,$  we denote the horizontal restriction of $\alpha$ by $e$ by $(e {}_*|\alpha). $ Similarly, if $e$ is a vertical arrow such that $e \leq \alpha \hcod,$ we denote the horizontal corestriction of $\alpha$ by $e$ by $(\alpha |_* e).$
\end{itemize}
\item \label{verind} $(\hor(\mathcal{G}),\dbl(\mathcal{G}))$ is an inductive groupoid.
\begin{itemize}
\item We denote the composition in this inductive groupoid -- the vertical composition from $\dbl(\mathcal{G})$ -- with $\bullet.$ We denote the partial order on this groupoid as $\lesssim.$  If $e$ and $f$ are vertical identity cells (horizontal arrows), we denote their meet as $e\vmeet f.$ For a cell $\alpha\in \dbl(\mathcal{G})$ and a horizontal arrow $e \in \hor(\mathcal{G})$ such that $e\lesssim \alpha \vdom,$  we denote the horizontal restriction of $\alpha$ by $e$ by $[e{}_*| \alpha]. $ Similarly, if $e$ is a horizontal arrow such that $e \lesssim \alpha \vcod,$ we denote the horizontal corestriction of $\alpha$ by $e$ by $[\alpha |_* e].$
\end{itemize}
\item \label{restcomp} When defined, vertical  (horizontal) composition and horizontal (vertical) (co)restriction must satisfy the middle-four interchange. For example, if $a$ and $b$ are double cells and $f$ and $g$ are horizontal arrows with $a\circ b$ defined, $f\circ g$ defined, $f \lesssim a\vcod$ and $g\lesssim b\vcod,$ then
\[ [a\circ b |_* f\circ g] = [a|_* f] \circ [b|_* g]. \]
\item \label{compmeet} When defined, vertical  (horizontal) composition and horizontal (vertical) meet must satisfy the middle-four interchange. For example, if $e,$ $f,$ $g$ and $h$ are horizontal arrows with $e\circ g$ and $f\circ h$ defined, then
\[ (e\vmeet f)\circ (g\vmeet h) = (e\circ g) \vmeet (f\circ h).\]
\item \label{restmeet} When defined, vertical  (horizontal) meet and horizontal (vertical) (co)restriction must satisfy the middle-four interchange. For example, if $e$ and $g$ are vertical arrows and $f$ and $h$ are objects with $f\lesssim e\vcod$ and $h\lesssim g\vcod,$ then
\[ [e\hmeet f |_* g\hmeet h] = [e|_* f] \hmeet [g|_* h]. \]
\item \label{restcomm} When defined, vertical (co)restriction and horizontal (co)restriction must satisfy the middle-four interchange. For example, if $a$ is a double cell, $f \lesssim a\vcod$ is a horizontal arrow, $g \leq a\hcod$ is a vertical arrow and $x = f\hcod \wedge g \vcod$ is an object, then
\[  ([a|_*  f] |_*  [g|_* x] ) = [ (a|_* g) |_*  (f|_* x) ] \]
\item \label{meetcomm} Vertical meet and horizontal meet must satisfy the middle-four interchange law; if $e,$ $f,$ $g$ and $h$ are both vertical and horizontal arrows, then
\[(e \hmeet f) \vmeet (g\hmeet h) = (e \vmeet g) \hmeet (f\vmeet h).\] 
\item \label{dommeet} When defined, vertical (horizontal) (co)domain must be functorial with respect to the horizontal (vertical) meet. For example, if $e$ and $f$ are horizontal arrows, then
\[ (e\vmeet f)\hdom = e\hdom \vmeet f\hdom.\]
\item \label{domrest} When defined, vertical (horizontal) (co)domain must be functorial with respect to the horizontal (vertical) (co)restriction. For example, if $a$ is a double cell and $f$ is a vertical arrow,  then
\[ (a |_* f)\hdom = (a\hdom |_* a\hdom).\qedhere\] 
\end{enumerate}
The reader may find the full list of conditions, as well as some diagrams to help understand them, in Appendix \ref{appendixa}.}
\end{definition}
\begin{definition} \emph{
Let $\mathcal{G}$ and $\mathcal{G'}$ be double inductive groupoids. A \emph{double inductive functor} $f:\mathcal{G}\rightarrow \mathcal{G'}$ is a double functor whose vertical arrow, horizontal arrow and double cell functions preserve all  partial orders and meets.}
\end{definition}
\begin{notation} \emph{
Consider a double inductive functor $f:\mathcal{G}\rightarrow \mathcal{G'}.$ We denote its object function by $f_0,$ its vertical arrow function by $f_v,$ its horizontal arrow function by $f_h$ and its double cell function by $f_d.$}
\end{notation}
%
%
\begin{construction} \label{digtodis} \emph{
Given a double inductive groupoid 
\[\mathcal{G} = (\obj(\mathcal{G}),\ver(\mathcal{G}),\hor(\mathcal{G}),\dbl(\mathcal{G})),\]
 we construct a double inverse semigroup $\cdis(\mathcal{G}) = (S,\circledcirc  ,\odot  )$ as follows:
\begin{enumerate}[--]
\item Its elements are the double cells of $\mathcal{G};$ $S = \dbl(\mathcal{G}).$
\item For any $a,b\in S,$ define
\[ a \circledcirc   b = (a|{}_* a\mathrm{hcod} \wedge_h b\mathrm{hdom}) \circ (a\mathrm{hcod} \wedge_h b\mathrm{hdom} {}_*| b) \]
\item For any $a,b\in S,$ define
\[ a \odot   b = [a|{}_* a\mathrm{vcod} \wedge_v b\mathrm{vdom}] \bullet [a\mathrm{vcod} \wedge_v b\mathrm{vdom} {}_*| b] \qedhere\] 
\end{enumerate} }
\end{construction}
It follows from Lawson's analogous Construction \ref{igtois} that each of the operations above are independently inverse semigroup operations. That these two inverse semigroup operations satisfy the middle-four interchange law, however, has yet to be verified. This verification is straight forward -- but  tedious -- and requires all properties of double inductive groupoids.  A full proof may be found in Appendix \ref{appendixb}; however, we will provide some visuals which represent the main idea of the verification, which is this: break the 4-fold semigroup product into a 4-fold composite of double cells, use the middle-four interchange property of the double inductive groupoids and then reassemble this composite into the desired  4-fold semigroup product.
\begin{figure}[H]
\centering
\includegraphics[scale=0.45]{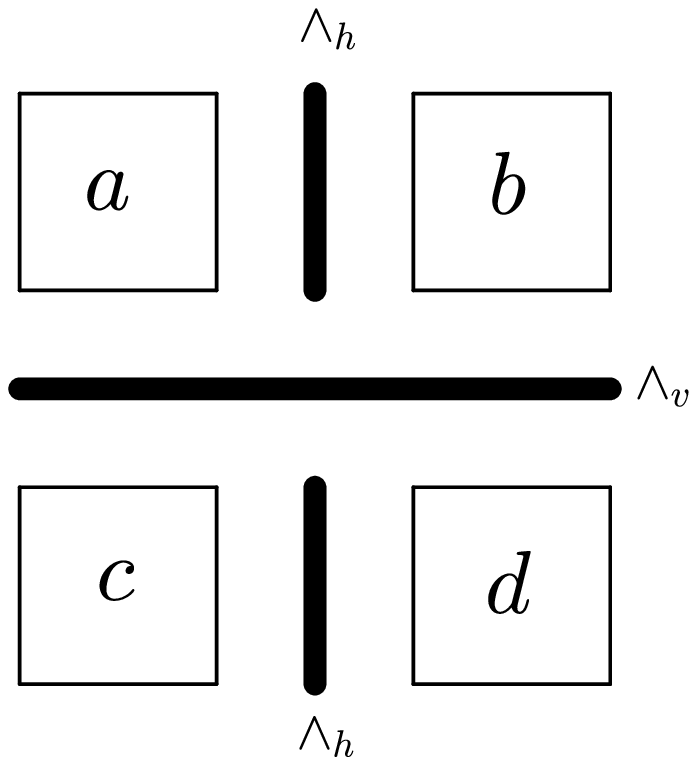} $\Longrightarrow$
\includegraphics[scale=0.45]{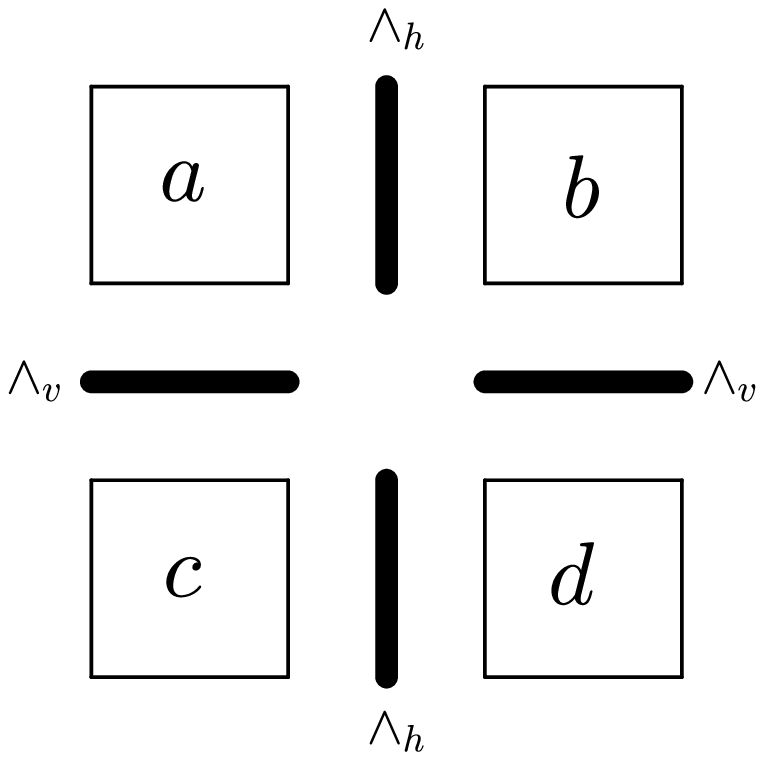} $\Longrightarrow$
\includegraphics[scale=0.45]{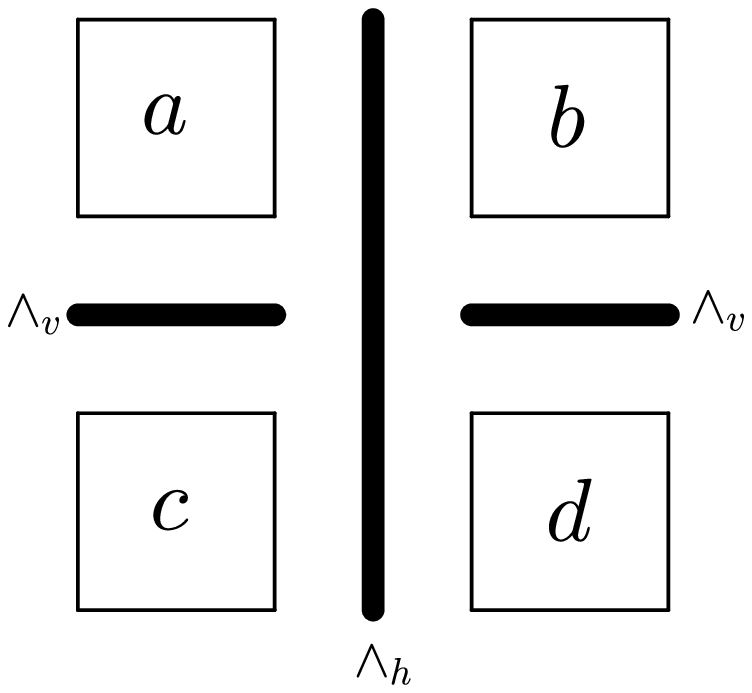}
\end{figure}
\begin{theorem}
If $\mathcal{G}$ is a double inductive groupoid, then $\cdis(\mathcal{G}),$ as constructed in Construction \ref{digtodis}, is a double inverse semigroup. \qed
\end{theorem}
%
%
\begin{construction} \label{distodig} \emph{
Given a double inverse semigroup $(S,\circledcirc  ,\odot  ),$ we construct a double inductive groupoid 
\[ \cdig(S) = (\obj(\cdig(S)), \ver(\cdig(S) ), \hor(\cdig(S) ), \dbl(\cdig(S) ))\]
 as follows:
\begin{enumerate}[--]
\item $\obj(\cdig(S)) = E(S,\circledcirc  )\cap E(S,\odot  ).$
\item $\ver(\cdig(S) ) = E(S,\circledcirc  ).$
\item $\hor(\cdig(S) ) = E(S, \odot  ).$
\item $\dbl(\cdig(S) ) = S(\circledcirc  ,\odot  ).$ Let $ a,b$ be any two double cells.
\begin{itemize}
\item We define $a\hdom = a\circledcirc   a^\circledcirc $ and $a\hcod = a^\circledcirc   \circledcirc   a.$ Whenever $a\hcod = b\hdom,$ the horizontal composite is defined as $a\circ b = a \circledcirc   b.$ Define a horizontal partial order $\leq$ by $a\leq b$ if and only if $a = e \circ  b = e\circledcirc  b$ for some vertical arrow $e.$ The horizontal meet of two vertical arrows $e$ and $f$ is defined to be $e\hmeet f = e \circledcirc   f.$ Note that since vertical arrows are contained in $S(\circledcirc  ,\odot  ),$ this $\leq$ is also a partial order on the vertical arrows $E(S,\circledcirc  ).$ If we have a vertical arrow $e\leq a\hdom,$ we define $(a|_*e) = a \circledcirc   e$ and if $e\leq a \hcod,$ we define $(e{}_*|a) = e \circledcirc  a.$
\item We define $a\vdom = a\odot   a^\odot  $ and $a\vcod = a^\odot   \odot   a.$ Whenever $a\vcod = b\vdom,$ the vertical composite is defined as $a\bullet b = a \odot   b.$ Define a vertical partial order $\lesssim$ by $a\lesssim b$ if and only if $a = e\bullet b = e \odot   b$ for some horizontal arrow $e.$ The vertical meet of two horizontal arrows $e$ and $f$ is defined to be $e\vmeet f = e \odot   f.$ Note that since horizontal arrows are contained in $S(\circledcirc  ,\odot  ),$ this $\lesssim$ is also a partial order on the horizontal  arrows $E(S,\odot  ).$ If we have a horizontal arrow $e\lesssim a\vdom,$ we define $[a|_*e] = a \odot   e$ and if $e\lesssim a \vcod,$ we define $[a{}_*|e] = e \odot   a.$ \qedhere
\end{itemize}
\end{enumerate} }
\end{construction}
\begin{note} \emph{
\begin{itemize}
\item The intersection $E(S,\odot) \cap E(S,\circledcirc)$ is non-empty since, for any $a\in S,$ $(a\odot a^\odot) \circledcirc (a \odot a^\odot )^\circledcirc \in E(S,\odot) \cap E(S,\circledcirc).$
\item From Lawson's Construction \ref{istoig}, it follows that $$(\ver(\cdig(S)), \dbl(\cdig(S)) ) \mbox{ and }(\hor(\cdig(S)), \dbl(\cdig(S)) )$$ are both inductive groupoids with the orders, meets, and (co)restrictions defined as above.
\item We have a groupoid structure in both directions and the horizontal and vertical compositions are defined by the horizontal and vertical semigroup products, respectively. These semigroup products satisfy the interchange law and thus so do the compositions. That is, if $S$ is a double inverse semigroup, then $\cdig(S),$ as constructed, is indeed a well defined double groupoid.
\item That $\cdig(S)$ satisfies the remaining properties of double inductive groupoids is a simple verification which is especially simplified by the commutativity of double inverse semigroups established in Theorem \ref{inversecommutethm}. For a full proof, see \cite[Lemmas 7.3.2 -- 7.3.11]{dewolf13}.
\end{itemize} }
\end{note} 
We summarise the preceding construction and its notes in the following theorem:
\begin{theorem}
If $S(\circledcirc , \odot )$ is a double inverse semigroup, then $\cdig(S),$ as constructed in Construction \ref{distodig}, is a double inductive groupoid. \qed
\end{theorem}
%
%
%
\begin{notation} \emph{
We denote the category of double inductive groupoids with double inductive functors as $\cat{DIG}$ and we denote the category of double inverse semigroups with double semigroup homomorphisms as $\cat{DIS}.$}
\end{notation} 
Our next goal is to prove that the constructions just introduced give rise to functors forming an isomorphism of the categories $\cat{DIG}$ and $\cat{DIS}.$
We will make use of a simple, yet important, observation which follows directly from the definition of double semigroup homomorphism:
\begin{note} \label{idemnote} \emph{
Let $\varphi: (S,\circledcirc ,\odot ) \rightarrow (S',\circledcirc ',\odot ')$ be a double semigroup homomorphism. If $e\in E(S,\odot )$ is an idempotent with respect to $\odot ,$ then $e\varphi = (e\odot  e)\varphi = e\varphi \odot ' e\varphi$ and $e\varphi \in E(S',\odot ')$ and thus $E(S,\odot ) \subseteq E(S,\odot )\varphi.$ Similarly, $E(S,\circledcirc ) \subseteq E(S,\circledcirc )\varphi.$ This tells us that $\varphi$ preserves idempotents in both directions. }
\end{note}
We will now define the functors of which the isomorphism comprises.
\begin{definition} \emph{ We define a functor $F: \cat{DIG} \rightarrow \cat{DIS}$ with the following data:
	\begin{itemize}
	\item On objects: For any double inductive groupoid $\mathcal{G},$ define $\mathcal{G}F = \cdis(\mathcal{G}),$
	as defined in Construction \ref{digtodis}. Recall that $\cdis(\mathcal{G}) = \dbl(\mathcal{G})$ with products defined, 
	for any $a,b\in \cdis(\mathcal{G}),$ as
	\begin{align*}
		a \circledcirc   b &= (a|{}_* a\mathrm{hcod} \wedge_h b\mathrm{hdom}) \circ (a\mathrm{hcod} \wedge_h b\mathrm{hdom} {}_*| b) \\
		a \odot   b &= [a|{}_* a\mathrm{vcod} \wedge_v b\mathrm{vdom}] \bullet [a\mathrm{vcod} \wedge_v b\mathrm{vdom} {}_*| b] 
	\end{align*}
	\item On arrows: For any double inductive functor $f:\mathcal{G}\rightarrow \mathcal{G'},$ define 
	$fF:\cdis(\mathcal{G})\rightarrow \cdis(\mathcal{G'})$ to be the double cell function $f_d$ of $f.$ \qedhere
	\end{itemize} }
\end{definition}
\begin{proposition}
The functor $F: \cat{DIG} \rightarrow \cat{DIS}$ as defined above is indeed a functor.
\end{proposition}
\begin{proof}
	Since the arrow function of $F$ returns the double cell function of a double functor, functoriality is immediate from the definition of functor composition.
	One also notes that, for any $a,b\in \cdis(\mathcal{G}),$ we have 
	\begin{align*}
		( a\odot  b) fF &= (a\odot  b)f_d \\
		&= \Big([a|{}_* a\mathrm{vcod} \wedge_v b\mathrm{vdom}] \bullet [a\mathrm{vcod} \wedge_v b\mathrm{vdom} {}_*| b]\Big)f_d \\
		&= \Big([a|{}_* a\mathrm{vcod} \wedge_v b\mathrm{vdom}] \Big) f_d \bullet' \Big([a\mathrm{vcod} \wedge_v b\mathrm{vdom} {}_*| b]\Big)f_d \\
				&\hspace{2pc} \mbox{ ($f_d$ preserves composition)} \\
		&= [af_d |{}_* (a\mathrm{vcod} \wedge_v b\mathrm{vdom})f_d] \bullet' [(a\mathrm{vcod} \wedge_v b\mathrm{vdom} )f_d{}_*| b f_d] \\
				&\hspace{2pc} \mbox{ ($f_d$ preserves (co)restrictions)}\\
		&= [af_d |{}_* a\mathrm{vcod}f_d \wedge_v b\mathrm{vdom}f_d] \bullet' [a\mathrm{vcod}f_d \wedge_v b\mathrm{vdom}f_d {}_*| b f_d]\\
				&\hspace{2pc} 	\mbox{ ($f_d$ preserves meets)} \\
		&= [af_d |{}_* a f_d \mathrm{vcod}\wedge_v b f_d \mathrm{vdom}] \bullet' [a f_d \mathrm{vcod} \wedge_v b f_d\mathrm{vdom} {}_*| b f_d] \\
				&\hspace{2pc} 	\mbox{ ($f_d$ preserves (co)domains)}\\
		&= af_d \odot ' bf_d \\
		&= afF \odot ' bfF.
	\end{align*}
	Similarly, we have that $(a\circledcirc  b)fF = afF \circledcirc ' bfF.$ That is, $fF$ is justifiably a double semigroup homomorphism.
\end{proof}
\begin{definition} \emph{ We define a functor $F':\cat{DIS} \rightarrow \cat{DIG}$ with the following data:
	\begin{itemize}
	\item On objects: For any double inverse semigroup $S,$ define $SF' = \cdig(S),$ as defined in
	Construction \ref{distodig}. Recall that $\cdig(S)$ has the following data:
		\begin{enumerate}[--]
		\item $\cdig(S)_0 = E(S,\odot )\cap E(S,\circledcirc ).$
		\item $\ver(\cdig(S)) = E(S,\circledcirc ).$
		\item $\hor(\cdig(S)) = E(S,\odot ).$
		\item $\dbl(\cdig(S)) = S(\odot ,\circledcirc ).$
		\end{enumerate}
	\item On arrows: For any double semigroup homomorphism $\varphi: S\rightarrow S'$ between double inverse semigroups, define
	$\varphi F' : \cdig(S)\rightarrow \cdig(S')$ to be the double (inductive) functor with the following data:
		\begin{enumerate}[--]
		\item An object function defined to be $\varphi$ restricted to $E(S,\odot )\cap E(S,\circledcirc ).$
		\item A vertical arrow function defined to be $\varphi$ restricted to $E(S,\circledcirc ).$
		\item A horizontal arrow function defined to be $\varphi$ restricted to $E(S,\odot ).$
		\item A double cell function defined to be $\varphi.$ \qedhere
		\end{enumerate}
	\end{itemize} }
\end{definition}
\begin{proposition}
The functor $F':\cat{DIS} \rightarrow \cat{DIG}$ as defined above is indeed a functor.
\end{proposition}
\begin{proof}
The above defined object, vertical arrow and horizontal arrow functions are well-defined due to the fact that
	double semigroup homomorphisms preserve idempotents (see Note \ref{idemnote}).	
	The (double) functoriality of $F'$ follows from the fact that the vertical and horizontal arrow functions are defined as restrictions to sets of idempotents and that the double cell function is the identity. For example, if 
	\[(S,\odot,\circledcirc) \stackrel{\varphi}{\longrightarrow} (S',\odot',\circledcirc') \stackrel{\varphi'}{\longrightarrow} (S'',\odot'',\circledcirc'') \] 
	are composable double inverse semigroup homomorphisms, then the vertical arrow function of $(\varphi;\varphi')F'$ is 
	\[ (\varphi;\varphi)|_{E(S,\circledcirc)} = \varphi|_{E(S,\circledcirc)}; \varphi'|_{\varphi|_{E(S,\circledcirc)}} = \varphi|_{E(S,\circledcirc)} ;\varphi'|_{E(S,\circledcirc)\varphi} = \varphi|_{E(S,\circledcirc)} ;\varphi'|_{E(S',\circledcirc')}.\]
	It remains, then, to show that $\varphi F'$ is actually inductive (since each function of $\varphi F'$ is either $\varphi$ or a restriction of $\varphi,$
	we will write only $\varphi):$
	\begin{enumerate}[(a)]
	\item We check that $\varphi $ preserves all partial orders. If $a,b\in \cdig(S)$ with $a\leq b,$ then by definition $a = e\circ b = e \circledcirc  b$
				for some $e\in \ver( \cdig(S)) = E(S,\circledcirc ).$ Since $\varphi$ is a homomorphism, then, 
				\begin{align*}
					&\hspace{0.5pc} a = e\circledcirc  b \\
					\Longrightarrow \hspace{1pc}& a \varphi = (e\circledcirc  b)\varphi \\
					\Longrightarrow \hspace{1pc}& a\varphi = e\varphi \circledcirc ' b\varphi \\
					\Longrightarrow \hspace{1pc}& a\varphi = e\varphi \circ' b\varphi \\
				\end{align*}
			Since $\varphi$ preserves idempotents, $e\varphi$ is indeed a vertical arrow and thus $a\varphi \leq' b\varphi$ in $\cdig(S').$ 
			Similarly, if $a\lesssim b$ in $\cdig(S),$ then $a\varphi \lesssim' b\varphi$ in $\cdig(S').$
	\item Since $\varphi$ is a double inverse semigroup homomorphism, $\varphi $ preserves all meets. For example, if $e$ and $f$ are vertical arrows in $\cdig(S),$ then
	\[ (e\hmeet f)\varphi = (e\circledcirc f)\varphi = e\varphi \circledcirc' f\varphi = e\varphi \hmeet' f\varphi.\]
	\item We verify that $\varphi $ preserves all (co)restrictions. Let $a$ be a double cell and $e\leq a\hcod$ a vertical arrow in $\cdig(S).$ Then $a\varphi$ is a double cell of $\cdig(S')$ and, since $\varphi$ preserves orders, $e\varphi\leq' a\varphi\hcod.$ Also,
	\[ (e {}_* | a)\varphi = (a\circledcirc a)\varphi = e\varphi \circledcirc' a\varphi = ( e\varphi {}_*| a\varphi) \]
	with 
	\[ (e\varphi {}_*| a\varphi )\hcod = (e {}_*| a)\varphi\hcod = (e{}_*|a)\hcod \varphi = e\varphi. \]
	\end{enumerate}
	It can then be said that the functor $\varphi$ defined in the arrow function of $F'$ is indeed inductive.
\end{proof}
We may now state and prove the following theorem:
\begin{theorem}
The pair of functors 
\[
	\xymatrixcolsep{3pc}
	\xymatrix{
		\cat{DIG} \ar@<0.5pc>[r]^-{F} & \cat{DIS} \ar@<0.5pc>[l]^-{F'} }
\]
is an isomorphism of categories.
\end{theorem}
\begin{proof}
We check that these two functors compose to the identity functors. 

Object functions: We know by our construction that the elements of a double inverse semigroup  $S$ are exactly the double cells of $\cdig(S)$ and that the double cells of a double inductive groupoid $\mathcal{G}$ are exactly the elements of $\cdis(\mathcal{G}).$ Then it is the case that the elements of $\cdis(\cdig(S))$ are exactly the elements of $S$ and the double cells of $\cdig(\cdis(\mathcal{G}))$ are exactly the double cells of $\mathcal{G}.$

We show that the products of elements in $\cdis(\cdig(S))$ are the same as those in $S.$ If $a,b\in S,$ we consider the product $a\odot  b.$ In  $\cdis(\cdig(S)),$ this product is
\begin{align*}
	&[a|{}_* a\vcod \vmeet b\vdom ] \bullet [a\vcod \vmeet b\vdom ] \\
	&= a\odot  (a^\odot  \odot  a ) \odot  (b\odot  b^\odot  ) \odot  (a^\odot  \odot  a) \odot  (b\odot  b^\odot ) \odot  b \\
	&= a \odot  ((a^\odot  \odot  a) \odot  (a^\odot  \odot  a)) \odot  ((b\odot  b^\odot ) \odot  (b\odot  b^\odot )) \odot  b \\
	&= (a\odot  a^\odot  \odot  a) \odot  (b\odot  b^\odot  \odot  b) \\
	&= a \odot  b.
\end{align*}
Similarly, $a\circledcirc  b$ in $\cdis(\cdig(S))$ is the same as in $S.$ Since the elements and the products are the same, we can say that $\cdis(\cdig(S)) = S.$

It will finally be shown that the composites of double cells in  $\mathcal{G}FF' = \cdig(\cdis(\mathcal{G}))$ are the same as those in $\mathcal{G}.$ We will then be done since vertical and horizontal arrows can be considered as identity double cells for horizontal and vertical composition, respectively. For any double cells $a,b\in \mathcal{G},$ if the composite $a\bullet b$ exists, we know that 
\[ a \odot  b = [ a|{}_* a\vcod \vmeet b\vdom ] \bullet [ a\vcod \vmeet b\vdom {}_*| b] \]
in the double inverse semigroup $\cdis(\mathcal{G}).$ However, since the composite $a\bullet b$ exists, $a\vcod = b\vdom$ and so this product in $\cdig(\cdis(\mathcal{G})),$ then, becomes
\begin{align*}
[ a|{}_* a\vcod  ] \bullet [ b\vdom {}_*| b]  &=(  a \odot  a^\odot  \odot  a  ) \bullet ( b \odot  b^\odot  \odot   b) \\
&= a \bullet b.
\end{align*}
Similarly, $a\circ b$ in $ \cdig(\cdis(\mathcal{G}))$ is the same as in $\mathcal{G}$ and we are done. Since the double cells (and thus horizontal and vertical arrows) and the composites are the same, we can say that  $ \cdig(\cdis(\mathcal{G})) = \mathcal{G}.$ 

Arrow functions: If $f$ is any inductive functor, then $fF$ is the double cell function and thus $fFF'$ is a functor whose double cell function is indeed just the the double cell function of $f.$ The object, vertical arrow and horizontal arrow functions of $fFF'$ are also just the object, vertical arrow and horizontal functions of $f.$ For example, the vertical arrow function of $fFF'$ is the restriction of $fF$ to the idempotents of the horizontal operation in the given double inverse semigroup. However, these idempotents are exactly the vertical arrows and thus the restriction of $fF$ by the horizontal idempotents is exactly the vertical arrow function of $f.$

If $\varphi$ is any double semigroup homomorphism, then the double cell function of $\varphi F'$ is just $\varphi.$ Then $\varphi F'F = \varphi,$ since $\varphi F'F$ is defined to be the double cell function of $\varphi F'.$
\end{proof}
%
%
\section{Main Result}
\label{mainresult}
%
We will now use the established  isomorphism of categories between $\cat{DIS}$ and $\cat{DIG}$ to formulate a characterisation of double inverse semigroups. 
\begin{lemma}
Let $(S,\odot, \circledcirc)$ be a double inverse semigroup. For all $a,b\in E(S,\odot) \cap E(S,\circledcirc),$ 
\[ a\odot b = a\circledcirc b.\]
\end{lemma}
\begin{proof}
We first note that, since $a$ and $b$ are idempotent with respect to $\odot$ and $\circledcirc,$
\[ a\odot a = a =  a\circledcirc a \]
and 
\[ b\odot b = b =  b\circledcirc b. \]
Also, 
\begin{align*}
 (a\circledcirc b) \odot (a\circledcirc b) &= (a\odot a) \circledcirc (b\odot b) \\
 &= a \circledcirc b.
\end{align*}
Using these facts and the commutativity in a double inverse semigroup,
\begin{align*}
a\odot b &= (a\odot b)\circledcirc (a\odot b) \\
&= (a\odot b) \circledcirc (b\odot a) \\
&= (a \circledcirc b) \odot (b\circledcirc a) \\
&= (a \circledcirc b) \odot (a\circledcirc b) \\
&= a\circledcirc b. \qedhere
\end{align*}
\end{proof}

Recall that our isomorphism gives rise to the following three facts:
\begin{itemize}
\item For all double cells $a$ and $b,$ $a\leq b$ ($a\lesssim b,$  respectively) if and only if there is some vertical arrow $e$ such that $a = e\circledcirc b$ (there is some horizontal arrow $e$ such that $a = e \odot b,$  respectively) 
\item For all vertical arrows $e$ and $f,$ $e\hmeet f = e\circledcirc f$ (for all horizontal arrows $e$ and $f,$ $e\vmeet f = e\odot f,$  respectively). 
\item For all double cells $a$ and vertical arrows $e,$ $(e {}_*| a) = e \circledcirc a$ (for all double cells $a$ and horizontal arrows $e,$ $[e {}_*|a] = e\odot a,$  respectively). 
\end{itemize}
That is, we can consider the partial order, meets and (co)restrictions as semigroup products. Because, on objects, these semigroup products coincide, we have the following corollary:
\begin{corollary} \label{collapse}
Let $\mathcal{G}$ be a double inductive groupoid. For all objects $a,b\in \obj(\mathcal{G}),$
\[ a\leq b \mbox{ if and only if } a\lesssim b \]
and
\[ a\hmeet b = a \vmeet b. \eqno \qed\] 
\end{corollary} 
\begin{note} \emph{
In any double inductive groupoid $\mathcal{G},$ recall that the sets $\ver(\mathcal{G})$ and $\hor(\mathcal{G})$ are posets with partial orders $\lesssim$ and $\leq,$ respectively. Because objects in $\obj(\mathcal{G})$ can be identified with either vertical or horizontal arrows, $\obj(\mathcal{G})$ is also a poset with respect to both $\lesssim$ and $\leq.$ By Corollary \ref{collapse}, however, these partial orders collapse on $\obj(\mathcal{G})$ so that it is unambiguous to talk about $\obj(\mathcal{G})$ as a partially ordered set.}
\end{note}

Having established that the two order structures of a double inductive groupoid coincide on its objects, we seek to further study and describe the relationship between the order on the objects and the orders on both the horizontal and vertical arrows. 

Let $\mathcal{G}$ be a double inductive groupoid and let $\cdis(\mathcal{G})$ be its corresponding double inverse semigroup. Given a double cell (element of $\cdis(\mathcal{G}))$ $a,$ it will have the following form:
\[
\xymatrixcolsep{3pc}
\xymatrix{
A \ar[r]^{a \odot a^\odot} \ar[d]_{a \circledcirc a^\circledcirc} \ar@{}[rd]|{a} &B \ar[d]^{a^\circledcirc \circledcirc a} \\
C \ar[r]_{a^\odot \odot a} & D }
\]
For convenience, let $a_h = a \odot a^\odot, a_v=a \circledcirc a^\circledcirc.$ By Theorem \ref{inversecommutethm}, both inverse semigroup products are commutative and $a,$ then, can be written as
\[
\xymatrixcolsep{3pc}
\xymatrix{
A \ar[r]^{a_h} \ar[d]_{a_v}|{\bullet} \ar@{}[rd]|{a} &B \ar[d]^{a_v}|{\bullet} \\
C \ar[r]_{a_h} & D }
\]
Since we can also write the domains and codomains of vertical and horizontal arrows as semigroup properties, commutativity then implies that
\[A = a_h\hdom = a_h\hcod = B = a_v \vdom = a_v\vcod = D = a_h \vcod = C. \]
We conclude that every double cell $a\in \dbl (\mathcal{G}),$ $a$ has the following form:
\[
\xymatrixcolsep{3pc}
\xymatrix{
A \ar[r]^{a_h} \ar[d]_{a_v}|{\bullet} \ar@{}[rd]|{a} &A \ar[d]^{a_v}|{\bullet} \\
A \ar[r]_{a_h} & A }
\]

Let $\mathcal{G}$ be a double inductive groupoid and let $A$ be an object of $\mathcal{G}.$ Then there is a natural collection of double cells
\[ (A)\mathcal{S}_\mathcal{G} = \left\{ a\in \dbl(\mathcal{G}) \Bigg| \vcenter{\hbox{\xymatrix{ A \ar@{}[rd]|a \ar[r]^{a_h} \ar[d]_{a_v}|\bullet & A\ar[d]^{a_v}|\bullet \\ A \ar[r]_{a_h} & A }}}  \right\}\]

Though definitely a double groupoid, it is not immediately obvious that $(A)\mathcal{S}_\mathcal{G}$ is a double \emph{inductive} groupoid. That is, it could be possible that meets or (co)restrictions may not be well defined on $(A)\mathcal{S}_\mathcal{G}$ (the meet of two arrows in $(A)\mathcal{S}_\mathcal{G}$ may not be in $(A)\mathcal{S}_\mathcal{G},$ for example). The following proposition, however, allows us to properly call them one-object double inductive groupoids:

\begin{proposition}
For each object $A\in \obj(\mathcal{G}),$ the above-defined collection $(A)\mathcal{S}_\mathcal{G}$ is a one-object double inductive groupoid.
\end{proposition}
\begin{proof}
These objects are subobjects of double groupoids and are thus double groupoids themselves. We now check the following properties:
\begin{enumerate}[(i)]

\item Vertical meets of horizontal arrows in $(A)\mathcal{S}_\mathcal{G}$ are again horizontal arrows in $(A)\mathcal{S}_\mathcal{G}.$
\item Horizontal meets of vertical arrows in $(A)\mathcal{S}_\mathcal{G}$ are again vertical arrows in $(A)\mathcal{S}_\mathcal{G}.$
\item Horizontal (co)restrictions of double cells in $(A)\mathcal{S}_\mathcal{G}$ by vertical arrows in $(A)\mathcal{S}_\mathcal{G}$ are again double cells of $(A)\mathcal{S}_\mathcal{G}.$
\item Vertical (co)restrictions of double cells in $(A)\mathcal{S}_\mathcal{G}$ by horizontal arrows in $(A)\mathcal{S}_\mathcal{G}$ are again double cells of $(A)\mathcal{S}_\mathcal{G}.$
\end{enumerate}

To prove (i), let $f$ and $g$ be horizontal arrows in $(A)\mathcal{S}_\mathcal{G}.$ We must show that the horizontal domain and codomain of $f\vmeet g$ are both $A.$ This is true since, in $\mathcal{G},$ we have the property that meets preserve domains and codomains. That is,
\[ (f\vmeet g)\hdom = f\hdom \vmeet g\hdom = A \vmeet A = A. \]
Similarly, $(f\vmeet g)\hcod = A.$

The proof of (ii) is similar to that of (i).

To prove (iii), we note that if we have a cell $a\in (A)\mathcal{S}_\mathcal{G}$ and a vertical arrow $e\leq a\hdom$ in $(A)\mathcal{S}_\mathcal{G},$ then the domain of the restriction $(e{}_*|a)$ is $e.$ Since $e\in (A)\mathcal{S}_\mathcal{G},$ the domain and codomain of $e$ are both $A.$ Because all four corners of a double cell are equal, all four corners of $(e{}_*|a)$ are $A$ and thus, this is an element of $(A)\mathcal{S}_\mathcal{G}.$

The proof of (iv) is similar to that of (iii).

These one-object double groupoids, then, are closed under meets and (co)restrictions. Because they are subobjects of a double inductive groupoid, they must satisfy all the axioms of a double inductive groupoid. Therefore, these one-object double groupoids are indeed one-object double inductive groupoids.
\end{proof}

By our isomorphism of categories, we can consider these one-object double inductive groupoids as a special class of double inverse semigroups whose operations share only one idempotent. Recall that an Abelian group $(A,+)$ can be considered as an improper double inverse semigroup $(A,+,+).$ Since groups have only one idempotent, the double inductive groupoid corresponding to $A$ will have only one object, one vertical arrow and one horizontal arrow. This motivates the proof of the following Proposition, which shows that we can conversely think of one-object double inductive groupoids as Abelian groups.

\begin{proposition}
If $A$ is a one-object double inductive groupoid, then $A$ is an Abelian group.
\end{proposition}
\begin{proof}
We first recall that in any double inductive groupoid, horizontal composition and vertical meets of horizontal arrows satisfy the middle-four interchange law. That is,
\[ (f\circ g) \vmeet (f' \circ g') = (f \vmeet f') \circ (g\vmeet g').\]
We note that  $a \vmeet b = a$ implies that $a\lesssim b$ and thus preservation of meets in this way implies the following law about preserving the vertical partial order:
\[ f \lesssim f', g\lesssim g' \mbox{ implies } f\circ g \lesssim f'\circ g' \]
Of course, in a single-object double inductive groupoid, all horizontal arrows have the same domain and codomain and are thus guaranteed to be  composable.

We show now that there is only one horizontal arrow, namely the identity $\id.$ Suppose that $f$ and $g$ are two horizontal arrows. Since the horizontal arrows form a meet-semilattice with respect to the vertical order, the meet of $f$ and $g$ exist and $f\vmeet g \lesssim f, g.$ Thus by preservation of the vertical partial order and inverses by horizontal composition,
\[ (f\vmeet g)^{-1} \circ f \lesssim f^{-1} \circ f = \id\]
and \[\id = (f\vmeet g)^{-1} \circ (f\vmeet g) \lesssim (f\vmeet g)^{-1} \circ f. \]
And thus $ f = f\vmeet g.$ Similarly $g = f\vmeet g$ and therefore $f = g (= \id).$ That is, the only horizontal arrow is $\id.$ Similarly, there is only the vertical identity arrow.

Having only one vertical and horizontal arrow, every pair of double cells in $A$ is composable in either direction. More specifically, each of the horizontal/vertical restrictions are trivial and thus the inverse semigroup operations reduce to the compositions in the double inductive groupoid associated with $A.$ The compositions, however, are group operations and $A$ is therefore an Abelian group by Eckmann-Hilton.
\end{proof}
If $A\in \obj(\mathcal{G})$ and $a\in (A)\mathcal{S}_\mathcal{G}$ and $e\leq A$ is an object (i.e., is both a horizontal and vertical arrow), then we know that the unique restriction of $a$ to $e$ is in $(e)\mathcal{S}_\mathcal{G}$ and has the form
\[
\xymatrixcolsep{5pc}
\xymatrix{
e  \ar[r]^{(e {}_*| a_h)} \ar[d]_{(e {}_*| a_v)} \ar@{}[rd]|{(e {}_*| a)} &e \ar[d]^{(e {}_*| a_v)} \\
e \ar[r]_{(e {}_*| a_h)} &e }
\]
We now consider the following map between $(A)\mathcal{S}_\mathcal{G}$ and $(e)\mathcal{S}_\mathcal{G}:$
\begin{align*}
\varphi_{e\leq A}: (A)\mathcal{S}_\mathcal{G} &\rightarrow (e)\mathcal{S}_\mathcal{G} \\
 a& \mapsto {(e {}_*| a)}
\end{align*}
If $a,b\in (A)\mathcal{S}_\mathcal{G},$ then
\begin{align*}
(a\bullet b)\varphi_{e\leq A} &= (e {}_*| a\bullet b) \\
&= {(e {}_*| a)} \bullet {(e {}_*| b)} \\
&= (a)\varphi_{e\leq A} \bullet (b)\varphi_{e\leq A}.
\end{align*}
That is, $\varphi_{e\leq A}: (A)\mathcal{S}_\mathcal{G} \rightarrow (e)\mathcal{S}_\mathcal{G}$ is an Abelian group homomorphism.
The discussion above shows that associated to any double inductive groupoid $\mathcal{G},$ we have a presheaf of Abelian groups, denoted
\[ \mathcal{S}_\mathcal{G} : \obj(\mathcal{G})^\mathrm{op} \rightarrow \cat{Ab}. \]

On objects we can send  $A$ to $(A)\mathcal{S}_\mathcal{G}$ and send an arrow $A\leq B$ to the Abelian group homomorphism $\varphi_{A\leq B}: (B)\mathcal{S}_\mathcal{G} \rightarrow (A)\mathcal{S}_\mathcal{G}$ (as described above) between $(B)\mathcal{S}_\mathcal{G}$ and $(A)\mathcal{S}_\mathcal{G}.$ 

We will establish an isomorphism of categories between the category of double inductive groupoids and presheaves of Abelian groups over meet-semilattices. We will show that the above construction induces a functor. We first, however, introduce the notion of a morphism between two presheaves of Abelian groups on meet-semilattices:

\begin{definition} \emph{
A morphism of presheaves of Abelian groups on meet-semilattices $P:L^\mathrm{op} \rightarrow \cat{Ab}$ and $P':({L'})^\mathrm{op} \rightarrow \cat{Ab}$ is an ordered pair
\[ (f,\{\psi_A\}_{A\in L} ): P \rightarrow P' \]
consisting of an order and meet preserving function $f:L\rightarrow L'$ and a family of group homomorphisms $\{\psi_A : AP\rightarrow (Af)P'\}$ indexed by the objects of $A$ such that, for any objects $A\leq B$ in $L,$ the following diagram commutes:
\[
\xymatrixcolsep{4pc}
\xymatrix{
AP \ar[d]_{\psi_A}  &BP \ar[l]_{\varphi_{A\leq B}}\ar[d]^{\psi_B} \\
(Af)P' &(Bf)P' \ar[l]^{\varphi'_{Af \leq' Bf}} }
\]}
\end{definition}
\begin{notation} \emph{
We denote the category of presheaves of Abelian groups on meet-semilattices with presheaf morphisms by $\cat{AbMeetSLatt}.$}
\end{notation}
We will now define the functors of which the isomorphism comprises.
\begin{definition}\emph{ We define a functor $F: \cat{DIG} \rightarrow \cat{AbMeetSLatt}$ with the following data:
	\begin{itemize}
	\item On objects: If $\mathcal{G}$ is a double inductive groupoid, define $\mathcal{G}F$ to be the presheaf $\mathcal{S}_\mathcal{G}$ of Abelian groups on the meet-semilattice $\obj(\mathcal{G})$ as detailed above (i.e., given by the restrictions).
	\item On arrows: Given a double inductive functor $f:\mathcal{G}\rightarrow \mathcal{G'},$
	define a morphism of presheaves $fF = (f_0,\{f_d|_{(A)\mathcal{S}_\mathcal{G}}\}_{A\in \obj(\mathcal{G})}),$ where
	\begin{itemize}
	\item $ f_0$ is the object function of $f.$
	\item $f_d|_{(A)\mathcal{S}_\mathcal{G}}$ is the double cell function of $f$ restricted to those cells who have all corners $A.$ 
	\end{itemize}
	\end{itemize}}
\end{definition}
\begin{note}	 \emph{
Given a double inductive functor $f:\mathcal{G}\rightarrow \mathcal{G'},$ $fF$ is indeed a morphism in $\cat{AbMeetSLatt}.$ By definition, $f_0$ is a morphism of meet-semilattices. If $A \leq B$ in $\obj(\mathcal{G}),$ we check that the following diagram commutes:
\[
\xymatrixcolsep{5pc}
\xymatrixrowsep{2pc}
\xymatrix{
A\mathcal{S}_\mathcal{G} \ar[d]_{f_d|_{A\mathcal{S}_\mathcal{G}}} &  B\mathcal{S}_\mathcal{G} \ar[d]^{f_d|_{B\mathcal{S}_\mathcal{G}}} \ar[l]_{\varphi_{A \leq B}} \\
(Af_0)\mathcal{S}_{\mathcal{G}'} & (Bf_0)\mathcal{S}_{\mathcal{G}'} \ar[l]^{\varphi'_{Af_0 \leq' Bf_0}}
}
\]
Since any object in a double category can be identified with an identity double cell and $f$ is a double functor, for any $b\in B\mathcal{S}_\mathcal{G},$
\begin{align*}
b (\varphi_{A\leq B} ; f_d|_{A\mathcal{S}_\mathcal{G}}) &= (A{}_*| b) f_d|_{A\mathcal{S}_\mathcal{G}} \\
&= (A f_d|_{B\mathcal{S}_\mathcal{G}} \, {}_*| \, b f_d|_{B\mathcal{S}_\mathcal{G}} )' \\
&= (A f_0 \, {}_*| \, b f_d|_{B\mathcal{S}_\mathcal{G}} )' \\
&=  b (f_d|_{B\mathcal{S}_\mathcal{G}} ; \varphi'_{Af_0 \leq' Bf_0})
\end{align*}
and thus the required diagram commutes. }
\end{note}
Being restrictions of double cell functions of double functors, each restriction of $f_d$ in the above definition preserves composition and identities. The following Proposition, then, follows.
\begin{proposition}
The functor $F: \cat{DIG} \rightarrow \cat{AbMeetSLatt}$ as defined above is indeed a functor.
\end{proposition}
\begin{definition} \emph{We define a functor $F':\cat{AbMeetSLatt} \rightarrow \cat{DIG}$ with the following data:
	\begin{itemize}
	\item On objects: If $P: L^\mathrm{op} \rightarrow \cat{Ab}$ is a presheaf of Abelian groups on a meet-semilattice, define a double inductive groupoid $\mathcal{G} = PF'$ with the following data:
	\begin{itemize}
	\item Objects: $\obj(\mathcal{G}) = L$
	\item Vertical/horizontal arrows: $\ver(\mathcal{G}) = \hor(\mathcal{G}) = \{e_{A} : A\in L\},$ where $e_A$ is the group unit of the Abelian group $AP$ for each $A$ in $L.$
	\item Double cells: $\dbl(\mathcal{G}) = \coprod_{A\in L} AP,$ the disjoint union of all Abelian groups $AP$ for $A$ in $L.$
	\item Since the vertical and horizontal arrows of $\mathcal{G}$ are the same, we will unambiguously not differentiate between directions in the following definitions of (co)domains, composition, meets and (co)restrictions. For each object $A\in L,$ there is exactly one horizontal/vertical arrow with (co)domain $A$ (the unit $e_A$ of the Abelian group $AP).$ There is therefore a one-to-one correspondence $A \leftrightarrow e_A : A\rightarrow A$ between the objects and vertical/horizontal arrows of $\mathcal{G}.$
	\item A double cell $a$ is contained in an Abelian group $AP$ for some $A\in L.$ We define $a \hdom = a\hcod = a\vdom = \vcod = e_A.$ The composite of two double cells, then, is defined when they are inside the same group and is defined to be the group product. Given two vertical/horizontal arrows $e_A$ and $e_B,$ we define their meet $e_A \wedge e_B = A\wedge B$ to be that from $L.$ If $a$ is a double cell and $e_u \leq e_A = a \hdom,$ then define the restriction of $a$ to $e_u$ to be $(e_u {}_*| a) = e_u *_u (a)\varphi_{u\leq A} = (a)\varphi_{u\leq A} .$ Corestrictions are similarly defined.
	\end{itemize}
	\item On arrows: If $(f,\{\psi_A\}_{A\in L}): P\rightarrow P'$ is a morphism of presheaves, define a double inductive functor \[g= (f,\{\psi_A\}_{A\in L})F': PF' \rightarrow P'F'\] with the following data:
		\begin{enumerate}[--]
		\item An object function: $g_0 = f.$
		\item A vertical/horizontal arrow function: For all $A\in L,$ 
		\[(e_A)g_v = (e_A)g_h = e'_{Af},\] the unit element of the group $(Af)P'.$
		\item A double cell function  $g_d$ defined by, for each double cell $a,$  \[(a)g_d = (a)\psi_{a\hdom\vdom}\]
		(evaluate each $a$ using the group homomorphism whose index is the object in the four corners of $a).$ \qedhere
		\end{enumerate}
	\end{itemize}}
\end{definition}
\begin{proposition}
The functor $F':\cat{AbMeetSLatt} \rightarrow \cat{DIG}$ as defined above is indeed a functor.
\end{proposition}
\begin{proof}
This is functorial, since it is composed of group homomorphisms. That is, composition (the group products) and identities are preserved. Also,
	\begin{enumerate}[(a)]
	\item This preserves all partial orders, since $f$ is an order preserving map.
	\item $F'$ preserves all meets. Recall that we have identified objects $A$ with the group unit $e_A$ in the group $AP.$ Since $f$ is meet-preserving, then,
	\begin{align*}
	(e_A \wedge e_B)g_v &= e'_{(e_A \wedge e_B)f} \\
	&= e'_{(A \wedge B)f} \\
	&= e'_{Af \wedge' Bf} \\
	&= Af \wedge' Bf \\
	&=  e_{Af} \wedge' e_{Bf} \\
	&=  (e_A )g_v \wedge' (e_B )g_v .
	\end{align*}
	\item $F'$ preserves (co)restrictions. Let $a\in AP$ be a double cell and let $e_u \leq e_A = a\hdom$ be a group identity. Note, then, that the restriction $(e_u {}_*| a)$ of $a$ to $e_u$ lives inside of $uP.$ From the definition of presheaf morphism, the $\psi$ maps commute with the $\varphi$ maps and thus
	\begin{align*}
	(e_u {}_*| a) g_d &= (e_u {}_*| a) \psi_u \\
	&= (a)(\varphi_{u\leq A}; \psi_u) \\
	&= (a)(\psi_A; \varphi_{uf \leq Af}) \\
	&= (a\psi_A)\varphi_{uf\leq Af} \\
	&= (e_{uf} {}_*| a\psi_A) \\
	&= (e_u g_d {}_*| ag_d)
	\end{align*}
	and $F'$ preserves restrictions. Similarly, $F'$ preserves corestrictions.
	\end{enumerate}
	It can then be said that the functor $g$ defined in the arrow function of $F'$ is indeed inductive.
\end{proof}
Finally, we may prove the following:
\begin{theorem} The pair of functors
\[
	\xymatrixcolsep{3pc}
	\xymatrix{
		\cat{DIG} \ar@<0.5pc>[r]^-{F} & \cat{AbMeetSLatt} \ar@<0.5pc>[l]^-{F'} }
\]
is an isomorphism of categories.
\end{theorem}
\begin{proof}
We check that these two functors compose to the identity functors. 
\begin{enumerate}[(1)]
\item Object functions: Given a presheaf $P:L^\mathrm{op}\rightarrow \cat{Ab},$ we check that $PF'F = P.$ First, $PF'$ is the double inductive groupoid with objects $L,$ vertical/horizontal arrows the identity elements and double cells the disjoint union of the Abelian groups $AP$ for $A\in L.$ When we send this double inductive groupoid into presheaves, we have a presheaf $P':L^\mathrm{op}\rightarrow \cat{Ab}$ whose object function sends elements of $L$ to its corresponding one-object double inductive groupoid in $PF'$ and sends arrows to group homomorphisms defined by the restriction in $PF'.$ We verify that these two presheaves agree on both their object and arrow functions:
\begin{enumerate}[(i)]
\item On objects: Given $B\in L,$
\[
BP' = \left\{ \vcenter{\hbox{\xymatrix{B \ar[r] \ar[d] \ar@{}[rd]|b & B \ar[d] \\ B \ar[r] &B}}} \bigg| b\in \dbl(G) = \coprod_{A\in L} AP \right\} = BP,
\]
since every double cell of this form must be in the group $BP.$
\item On arrows: Suppose that $B \leq C$ in $L.$ We then have a map $\varphi'_{B\leq C}:CP'\rightarrow BP'= CP \rightarrow BP$ given by the restriction in $PF',$ which is defined by evaluation of the homomorphism $\varphi_{B\leq C}:CP \rightarrow BP$ from $P.$  Then, for all $c\in CP,$
\[ (c)\varphi'_{B\leq C} = (B{}_*|c) = (c)\varphi_{B\leq C}.\]
\end{enumerate}

Conversely, suppose that we are given a double inductive groupoid $\mathcal{G}.$ We check that $\mathcal{G}FF' = \mathcal{G}.$ We know that $\mathcal{G}F$ is a presheaf $P: \obj(\mathcal{G})^\mathrm{op} \rightarrow \cat{Ab}.$ Sending this presheaf into double inductive groupoids, then, gives us a double inductive groupoid with objects $\obj(\mathcal{G}),$ vertical/horizontal arrows all identities and double cell the disjoint union of the groups in the image of $P.$ It was shown before, however, that a double inductive groupoid consists solely of cells that lie inside of these groups and have only identities for vertical and horizontal arrows. It is clear, then, that this is the same double inductive groupoid, or that $\mathcal{G}FF' = \mathcal{G}.$

\item Arrow Functions: Given a morphism $(f, \{\psi_A: AP \rightarrow (Af)P'\}: P \rightarrow P'$ of presheaves,  its image under $F'$ in $\cat{DIG}$ is the double functor whose object function is $f$ and whose double cell function is comprised of the $\psi_A$ maps. The image of this double functor under $F,$ then, is just the presheaf morphism consisting of $f:L\rightarrow L'$ and, for each $A\in L,$ the double cell function restricted to $A,$ which is equal to $\psi_A,$ and this composite is the identity.

Conversely, suppose that $f: \mathcal{G}\rightarrow \mathcal{G'}$ is a double inductive functor. Then $fF$ is the morphism of presheaves
\[(f_0, \{f_d|_{A\mathcal{S}_\mathcal{G}}: A\mathcal{S}_\mathcal{G} \rightarrow (Af_0)\mathcal{S}_\mathcal{G}\}):\mathcal{G}F \rightarrow \mathcal{G'}F. \]
The image of this morphism under $F',$ then will be the double inductive functor whose
\begin{enumerate}[--]
\item Object function is $f_0.$
\item Vertical and horizontal arrow functions map identities to the identities under $f_0.$ Since a double inductive groupoid contains only identities as vertical and horizontal arrows, these are the vertical and horizontal arrow functions from $f.$
\item Double cell function is defined by the double cell function of $f$ restricted to $A\mathcal{S}_\mathcal{G}$ for each object $A.$ Since every double cell in a $\mathcal{G}$ is contained in some $A\mathcal{S}_\mathcal{G},$ these double cell functions also coincide.
\end{enumerate}
That is, these double inductive functors are equal and $f = fFF'.$
\end{enumerate}

Having defined two functors whose composition is the identity functor in either way, we have completed the proof.
\end{proof}

We have shown that double inverse semigroups are exactly presheaves of Abelian groups on meet-semilattices. In particular, we have seen that double inverse semigroups consist of a collection of single object double groupoids (indexed by its idempotents) in which both the vertical and horizontal composition coincide (i.e., groups). Add to this Kock's result that double inverse semigroups commute and we have the following result:
\begin{theorem} \label{thm:main}
Double inverse semigroups are commutative and improper. That is to say that $(S,\odot, \circledcirc)$ is a double inverse semigroup if and only if both $\odot$ and $\circledcirc$ are commutative inverse semigroup operations with $\circledcirc = \odot.$ \qed
\end{theorem}
\begin{note}
A \emph{Clifford semigroup} is a completely regular inverse semigroup; that is, each element is in some subgroup of the semigroup. An inverse semigroup $S$ is a Clifford semigroup if it satisfies $xx^{-1} = x^{-1}x$ for all $x\in S$ \cite{carvalho03}. Theorem \ref{thm:main} says, then, that double inverse semigroups are commutative Clifford semigroups.
\end{note}
\begin{appendices}
  \renewcommand\thetable{\thesection\arabic{table}}
  \renewcommand\thefigure{\thesection\arabic{figure}}
  \section{Full Definition of Double Inductive Groupoids}
   \label{appendixa}
   \begin{definition} \emph{
A \emph{double inductive groupoid}
\[ \mathcal{G} = (\obj(\mathcal{G}),\ver(\mathcal{G}),\hor(\mathcal{G}),\dbl(\mathcal{G}), \leq, \lesssim)\]
 is a double groupoid (i.e., a double category in which every vertical and horizontal arrow is an isomorphism and each double cell is an isomorphism with respect to both the horizontal and vertical composition) such that
\begin{enumerate}[(i)]
\item $(\ver(\mathcal{G}),\dbl(\mathcal{G}))$ is an inductive groupoid.
\begin{itemize}
\item  We denote the composition in this inductive groupoid -- the horizontal composition from $\dbl(\mathcal{G})$ -- with $\circ.$  We denote the partial order on this groupoid as $\leq.$ If $e$ and $f$ are horizontal identity cells (vertical arrows), we denote their meet as $e\hmeet f.$ For a cell $\alpha\in \dbl(\mathcal{G})$ and a vertical arrow $e \in \ver(\mathcal{G})$ such that $e\leq \alpha \hdom,$  we denote the horizontal restriction of $\alpha$ by $e$ by $(e {}_*|\alpha). $ Similarly, if $e$ is a vertical arrow such that $e \leq \alpha \hcod,$ we denote the horizontal corestriction of $\alpha$ by $e$ by $(\alpha |_* e).$
\end{itemize}
\item $(\hor(\mathcal{G}),\dbl(\mathcal{G}))$ is an inductive groupoid.
\begin{itemize}
\item We denote the composition in this inductive groupoid -- the vertical composition from $\dbl(\mathcal{G})$ -- with $\bullet.$ We denote the partial order on this groupoid as $\lesssim.$  If $e$ and $f$ are vertical identity cells (horizontal arrows), we denote their meet as $e\vmeet f.$ For a cell $\alpha\in \dbl(\mathcal{G})$ and a horizontal arrow $e \in \hor(\mathcal{G})$ such that $e\lesssim \alpha \vdom,$  we denote the horizontal restriction of $\alpha$ by $e$ by $[e{}_*| \alpha]. $ Similarly, if $e$ is a horizontal arrow such that $e \lesssim \alpha \vcod,$ we denote the horizontal corestriction of $\alpha$ by $e$ by $[\alpha |_* e].$
\end{itemize}
\item When defined, vertical  (horizontal, respectively) composition and horizontal (vertical, respectively) (co)restriction must satisfy the middle-four interchange. Explicitly, 
\begin{enumerate}
\item $\big( a\bullet b |_* f\bullet g \big) = \big( a |_* f\big) \bullet \big(  b |_* g \big).$
\item $\big[ a\circ b |_* f\circ g \big] = \big[ a |_* f\big] \circ \big[  b |_* g \big].$
\item $\big(f\bullet g {}_*|   a\bullet b  \big) = \big( f  {}_*| a \big) \bullet \big(  g {}_*|  b\big).$ 
\item $\big[f\circ g  {}_*|  a\circ  b  \big] = \big[f   {}_*|  a \big] \circ \big[g   {}_*|  b  \big].$ 
\end{enumerate}
For example, for (a) to be well defined, we need that $a$ and $b$ are double cells and $f$ and $g$ are vertical arrows with $a\bullet b$ defined, $f\bullet g$ defined, $f \leq a\hcod$ and $g\leq b\hcod.$ The rule $\big( a\bullet  b |_* f\bullet g \big) = \big( a |_* f\big) \bullet \big(  b |_* g \big),$ visually:
\[
\xymatrixcolsep{7pc}
\xymatrixrowsep{3pc}
\vcenter{\hbox{\xymatrix{
{} \ar[r] \ar[d]|\bullet \ar@{}[rd]|{\big( a |_* f\big)} &{} \ar[d]^f|\bullet \\
{} \ar[r] \ar[d]|\bullet \ar@{}[rd]|{\big(  b |_* g \big)} &{} \ar[d]^g|\bullet \\
{} \ar[r] &{} 
} } } = 
\xymatrixcolsep{7pc}
\xymatrixrowsep{6pc}
\vcenter{\hbox{\xymatrix{
{} \ar[r] \ar[d]|\bullet \ar@{}[rd]|{\big( a\bullet b |_* f\bullet g \big)} &{} \ar[d]^{f\bullet g}|\bullet  \\
{} \ar[r] &{} 
} } } 
\]
\item When defined, vertical  (horizontal, respectively) composition and horizontal (vertical, respectively) meet must satisfy the middle-four interchange. 
Explicitly,
\begin{enumerate}[(a)]
\item $(e\vmeet f) \circ (g\vmeet h) = (e \circ g) \vmeet (f \circ h).$
\item $(e\hmeet f) \bullet (g\hmeet h) = (e \bullet g) \hmeet (f \bullet h).$
\end{enumerate}
For example, for rule (a) to be well defined, we need that $e,$ $f,$ $g$ and $h$ are horizontal arrows with $e\circ g$ and $f\circ h$ defined. The rule $(e\hmeet f) \bullet (g\hmeet h) = (e \bullet g) \hmeet (f \bullet h),$ visually:
\[
\xymatrixcolsep{3pc}
\xymatrixrowsep{3pc}
\vcenter{\hbox{\xymatrix{
{} \ar[d]_e|\bullet & {} \ar[d]^f|\bullet \\
{} \ar[d]_g|\bullet \ar@{}[r]|\hmeet &{} \ar[d]^h|\bullet \\
{} &{} }}} =
\vcenter{\hbox{\xymatrix{
{} \ar[d]^{e\hmeet f}|\bullet \\
{} \ar[d]^{g\hmeet h}|\bullet \\
{} }}}
\]
\item  When defined, vertical  (horizontal) meet and horizontal (vertical) (co)restriction must satisfy the middle-four interchange. 
Explicitly,
\begin{enumerate}[(a)]
\item $(e|_* f) \vmeet (g|_* h) = (e\vmeet g |_*  f\vmeet h).$
\item $[e|_* f] \hmeet [g|_* h] = [ e\hmeet g |_*  f\hmeet h].$
\item $(e {}_*| f) \vmeet (g {}_*| h) = (e\vmeet g  {}_*|  f\vmeet h).$
\item $[e {}_*| f] \hmeet [g {}_*| h] = [ e\hmeet g  {}_*|  f\hmeet h].$
\end{enumerate}
For example, for rule (b) to be well defined, we need that $e$ and $g$ are vertical arrows and $f$ and $h$ are objects with $f\lesssim e\vcod$ and $h\lesssim g\vcod.$ The rule $(e|_* f) \vmeet (g|_* h) = (e\vmeet g |_*  f\vmeet h),$ visually:
\[
\xymatrixcolsep{2.5pc}
\xymatrixrowsep{3pc}
\vcenter{\hbox{\xymatrix{
{} \ar[rr]^-{(e|{}_* f)} &{} \ar@{}[d]|\vmeet &{f}\\
{} \ar[rr]_-{(g|{}_* h)} &{} &{h}
}}} = 
\xymatrixcolsep{5pc}
\vcenter{\hbox{\xymatrix{
{} \ar[r]_-{(e\vmeet g |_*  f\vmeet h)} &{f\vmeet h} 
}}}
\]
\item  When defined, vertical (co)restriction and horizontal (co)restriction must satisfy the middle-four interchange. Explicitly,
\begin{enumerate}[(a)]
\item $ ([a|_*  f] |_*  [g|_* x] ) = [ (a|_* g) |_*  (f|_* x) ].$
\item $ [(a|_* g) |_*  (f|_* x) ] = ( [a|_* f] |_*  [g|_* x] ) .$
\item $ ([x {}_*|  g]  {}_*|  [f {}_*| a] ) = [ (x {}_*| f)  {}_*|  (g {}_*| a) ].$
\item $ [(x {}_*| f)  {}_*|  (g {}_*| a) ] = ( [x {}_*| g] | [f {}_*| a] ) .$
\end{enumerate}
The rule $ ([a|_*  f] |_*  [g|_* x] ) = [ (a|_* g) |_*  (f|_* x)],$ visually:
\[
\xymatrix{
{} \ar[rr] \ar[dd] & {} &{} \ar[dd] &{} \ar[dd]^g \\
{} &{a} &{} \ar@{}[r]|\geq &{} \\
{} \ar[rr] &{}  &{} &{g \vcod } \\
{} \ar[rr]_f &{}\ar@{}[u]|-*[@]{\lesssim} &{f\hcod}  &{x}\rlap{$=f\hcod \wedge g\vcod$} \ar@{}[l]|(.2)*[@]{\lesssim} \ar@{}[u]|-*[@]{\leq}  }
\]
\item  When defined, vertical meet and horizontal meet must satisfy the middle-four interchange law:
\[(e \hmeet f) \vmeet (g\hmeet h) = (e \vmeet g) \hmeet (f\vmeet h).\] 
\item  When defined, vertical (horizontal) (co)domain must be functorial with respect to the horizontal (vertical) meet. Explicitly,
\begin{enumerate}[(a)]
\item $(e \hmeet f) \vdom = e\vdom \hmeet f \vdom.$
\item $(e \hmeet f) \vcod = e\vcod \hmeet f\vcod.$
\item $(e\vmeet f) \hdom = e\hdom \vmeet f\hdom.$
\item $(e \vmeet f) \hcod = e\hcod \vmeet f\hcod.$
\end{enumerate}
The rule $(e \hmeet f) \vdom = e\vdom \hmeet f \vdom,$ visually:
\[
\xymatrix{
A \ar[d]_e & B \ar[d]^f & A\hmeet B \ar[d]^{e\hmeet f}\\
{} &{} &{}}
\]
\item When defined, vertical (horizontal) (co)domain must be functorial with respect to the horizontal (vertical) (co)restriction. Explicitly,
\begin{enumerate}[(a)]
\item $(a|_* e)\vdom = (a\vdom|_* e\vdom).$
\item $(a|_* e) \vcod = (a\vcod |_*  e\vcod).$
\item $(e {}_*| a) \vdom = (e\vdom  {}_*|  a\vdom).$
\item $(e {}_*| a) \vcod = (e\vcod  {}_*|  a \vcod).$
\item $[a|_* e]\hdom = [a\hdom |_* e\hdom].$
\item $[a|_* e] \hcod = [a\hcod |_*  e\hcod].$
\item $[e {}_*| a] \hdom = [e\vdom  {}_*|  a\hdom].$
\item $[e {}_*| a] \hcod = [e\hcod  {}_*|  a \hcod].$
\end{enumerate}
The rules $(a|_* e)\vdom = (a\vdom|_* e\vdom)$ and $(a|_* e)\hdom = (a\hdom|_* e\hdom),$ visually:
\[
\xymatrixcolsep{8pc}
\xymatrixrowsep{6pc}
\xymatrix{
{} \ar[d] \ar@{}[rd]|{(a|_* e)}\ar[r]^{(a\vdom|_* e\vdom)} &{} \ar[d]^e \\
{} \ar[r]_{(a\hdom|_* e\hdom)} &{} } \qedhere
\]
\end{enumerate} }
\end{definition}

  \section{Proof of the Interchange Law in $\cdis(\mathcal{G})$}
   \label{appendixb}
Given a double inductive groupoid $\mathcal{G},$ refer to Section \ref{doubleinductivegroupoids} for the construction of the double inverse semigroup $\cdis(\mathcal{G})$ based on $\mathcal{G}.$ We give here a full proof that the horizontal and vertical operations of $\cdis(\mathcal{G})$ satisfy middle-four interchange law:
\begin{proposition} \label{m4dig}
For all $a,b,c,d\in \cdis(\mathcal{G}),$ $(a\circledcirc b)\odot (c\circledcirc d) = (a\odot c)\circledcirc (c\odot d).$
\end{proposition}
Before proving this proposition, we will prove some smaller results that will more easily allow us to handle the mechanics of the proof. During the proof, we will find ourselves in the situation where we want to talk about the vertical composite of two vertical (co)restrictions that are themselves horizontal composites of horizontal (co)restrictions. This proposition allows us to break this nested composite into a four-fold product of double cells in $\mathcal{G}.$ 

For double cells $a,b,c,d\in \mathcal{G},$ define $u = a\hcod \hmeet b\hdom$ and $v = c\hcod \hmeet d \hdom.$ Consider, then, the following composite in $\cdis(\mathcal{G}):$
\begin{center}
\includegraphics[scale=0.6]{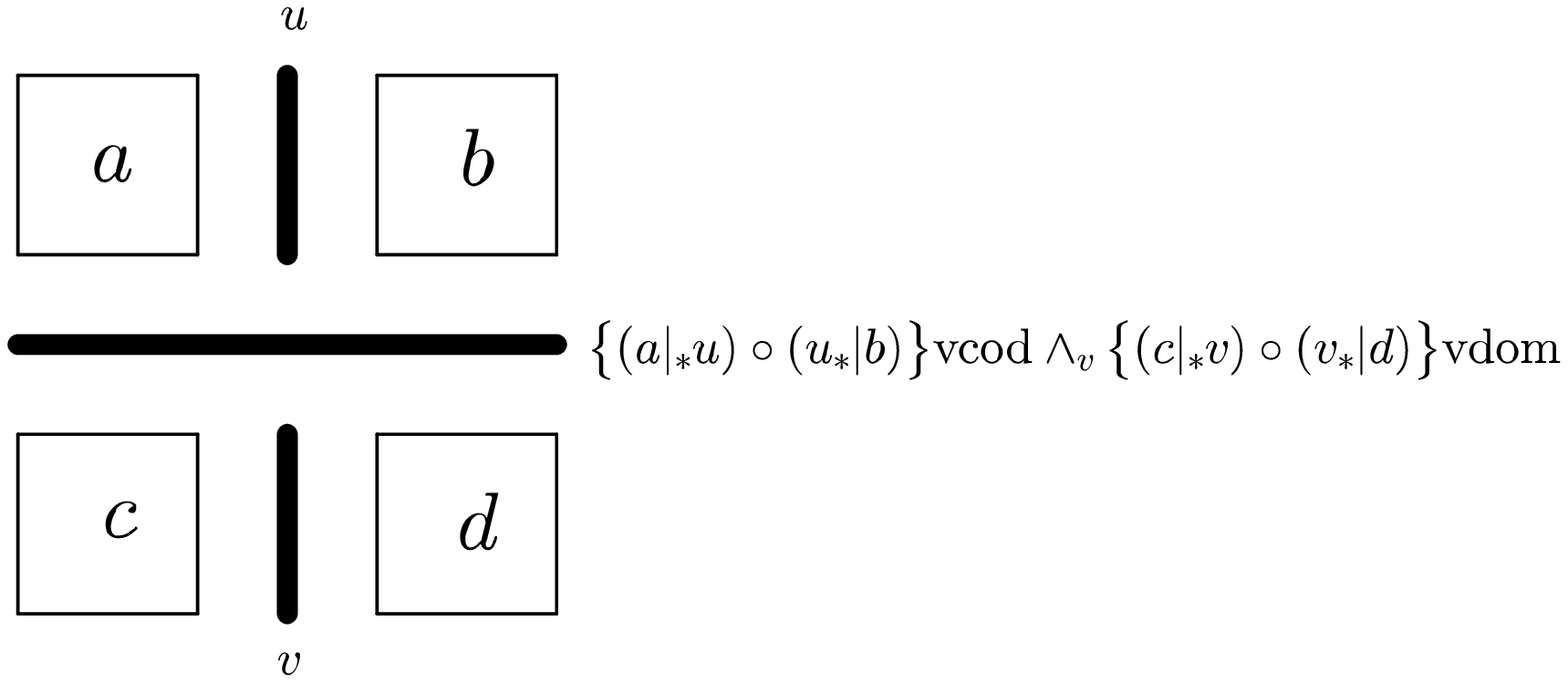}
\end{center}
\begin{proposition} \label{fact1}
\begin{enumerate}[(i)]
\item $\bigg[(a|_*u)\circ (u{}_*|b) \bigg|_* \big\{ (a|_*u)\circ (u{}_*|b)\big\}\vcod \vmeet \big\{(c|_*v)\circ (v{}_*|d)\big\}\vdom \bigg] \\= \bigg[ (a|_*u)\bigg|_* (a|_*u)\vcod \vmeet (c|_*v)\vdom \bigg] \circ \bigg[(u{}_*|b)\bigg|_* (u{}_*|b)\vcod \vmeet (v{}_*|d)\vdom \bigg] $
\item $\bigg[ \big\{ (a|_*u)\circ (u{}_*|b)\big\}\vcod \vmeet \big\{(c|_*v)\circ (v{}_*|d)\big\}\vdom \bigg|_* (c|_*v)\circ (v{}_*|d) \bigg] \\= \bigg[ (a|_*u)\vcod \vmeet (c|_*v)\vdom \bigg|_* (c|_*v) \bigg]\circ \bigg[ (u{}_*|b)\vcod \vmeet (v{}_*|d)\vdom \bigg|_* (v{}_*|d) \bigg]$
\end{enumerate}
\end{proposition}
\begin{proof}
We prove here only (i) -- the proof of (ii) is analogous. We apply first the functoriality of vertical domains and codomains with respect to horizontal (co)restrictions followed by the interchange law between vertical meets and horizontal composition. Finally, we apply the functoriality of vertical (co)restrictions with respect to horizontal composition.
\begin{align*}
&\bigg[(a|_*u)\circ (u{}_*|b) \bigg|_* \big\{ (a|_*u)\circ (u{}_*|b)\big\}\vcod \vmeet \big\{(c|_*v)\circ (v{}_*|d)\big\}\vdom \bigg] \\
&= \bigg[(a|_*u)\circ (u{}_*|b) \bigg|_* \big\{ (a|_*u)\vcod\circ (u{}_*|b)\vcod\big\} \vmeet \big\{(c|_*v)\vdom\circ (v{}_*|d)\vdom\big\}\bigg] \\
&= \bigg[(a|_*u)\circ (u{}_*|b) \bigg|_* \big\{ (a|_*u)\vcod\vmeet (c|_*v)\vdom\big\} \circ \big\{(u{}_*|b)\vcod \vmeet (v{}_*|d)\vdom\big\}\bigg] \\
&= \bigg[ (a|_*u)\bigg|_* (a|_*u)\vcod \vmeet (c|_*v)\vdom \bigg] \circ \bigg[(u{}_*|b)\bigg|_* (u{}_*|b)\vcod \vmeet (v{}_*|d)\vdom \bigg] 
\end{align*}
Note that $(a|_*u)\vcod\hcod = u\vcod$ and $(u{}_*|b)\vcod\hdom = u\vdom$ so that the composite $(a|_*u)\vcod\circ (u{}_*|b)\vcod$ is defined. A similar calculation shows that $(c|_*v)\vdom\circ (v{}_*|d)\vdom$ is defined. This justifies the second equality.

By the definition of meet, $ (a|_*u)\vcod \vmeet (c|_*v)\vdom \lesssim (a|_*u)\vcod$ so that the corestriction $\big[ (a|_*u)\big|_* (a|_*u)\vcod \vmeet (c|_*v)\vdom \big]$ is defined. Similarly, $\big[(u{}_*|b)\big|_* (u{}_*|b)\vcod \vmeet (v{}_*|d)\vdom \big]$ is defined and this justifies the third equality.
\end{proof}
The following Proposition is analogous to the previous in the sense that we are now in the following situation involving the horizontal composite of horizontal (co)restrictions consisting of vertical (co)restrictions. We can write these, too, as a four-fold product of double cells in $\mathcal{G}.$

For double cells $a,b,c,d\in \mathcal{G},$ define $f = a\vcod \vmeet c\vdom$ and $g=b\vcod \vmeet d\vdom.$ Consider, then, the following composite in $\cdis(\mathcal{G}):$
\begin{center}
\includegraphics[scale=0.6]{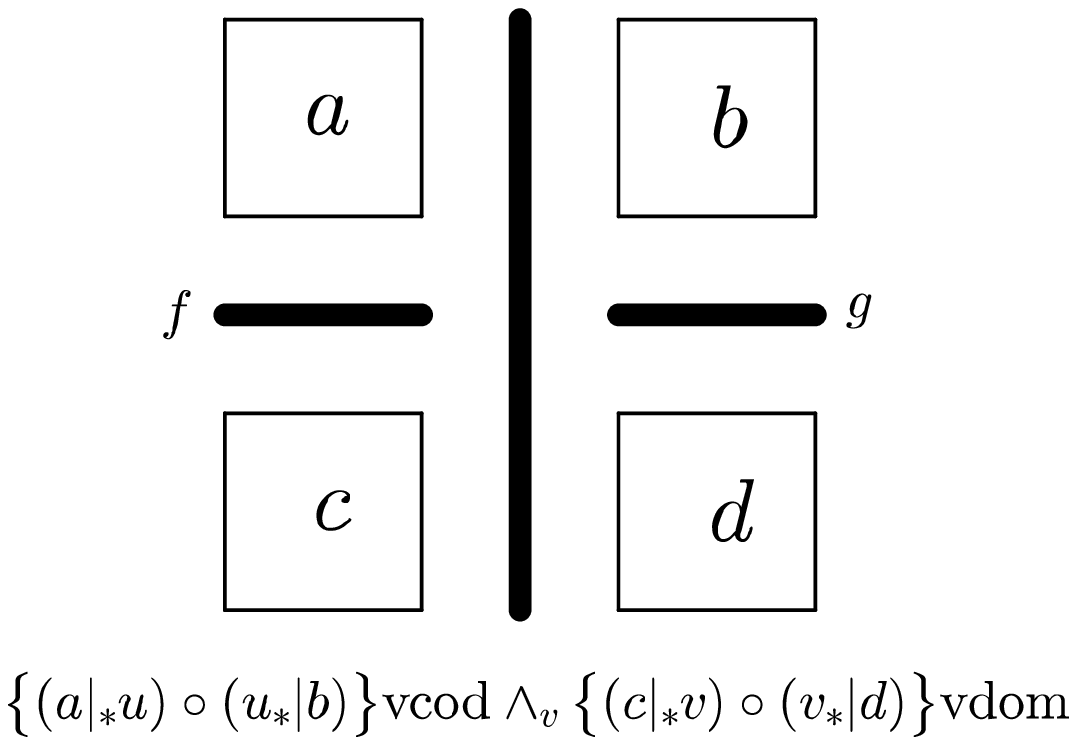}
\end{center}
\begin{proposition} \label{fact2}
\begin{enumerate}[(i)] \hfill
\item $\bigg([a|_*f]\bullet [f{}_*|c] \bigg|_* \big\{ [a|_*f]\bullet [f{}_*|c]\big\}\hcod \hmeet \big\{[b|_*g]\bullet [g{}_*|d]\big\}\hdom \bigg) \\= \bigg( [a|_*f]\bigg|_* [a|_*f]\hcod \hmeet [b|_*g]\hdom \bigg) \bullet \bigg([f{}_*|c]\bigg|_* [f{}_*|c]\hcod \hmeet [g{}_*|d]\hdom \bigg) $
\item $\bigg(\big\{ [a|_*f]\bullet [f{}_*|c]\big\}\hcod \hmeet \big\{[b|_*g]\bullet [g{}_*|d]\big\}\hdom{}_*\bigg| [b|_*g]\bullet [g{}_*|d]\bigg) \\= \bigg([a|_*f]\hcod \hmeet [b|_*g]\hdom {}_*\bigg| [b|_*g] \bigg)\bullet \bigg( [f{}_*|c]\hcod \hmeet [g{}_*|d]\hdom {}_*\bigg| [g{}_*|d]\bigg)$
\end{enumerate}
\end{proposition}
\begin{proof}
We prove here only (i) -- the proof of (ii) is analogous. We apply first the functoriality of horizontal  domains and codomains with respect to vertical (co)restrictions followed by the interchange law between horizontal meets and vertical composition. Finally, we apply the functoriality of horizontal (co)restrictions with respect to vertical composition.
\begin{align*}
& \bigg([a|_*f]\bullet [f{}_*|c] \bigg|_* \big\{ [a|_*f]\bullet [f{}_*|c]\big\}\hcod \hmeet \big\{[b|_*g]\bullet [g{}_*|d]\big\}\hdom \bigg) \\
&= \bigg([a|_*f]\bullet [f{}_*|c] \bigg|_* \big\{ [a|_*f]\hcod\bullet [f{}_*|c]\hcod\big\} \hmeet \big\{[b|_*g]\hdom\bullet [g{}_*|d]\hdom\big\} \bigg) \\
&= \bigg([a|_*f]\bullet [f{}_*|c] \bigg|_* \big\{ [a|_*f]\hcod\hmeet [b|_*g]\hdom\big\} \bullet \big\{[f{}_*|c]\hcod \hmeet [g{}_*|d]\hdom\big\} \bigg) \\
&= \bigg( [a|_*f]\bigg|_* [a|_*f]\hcod \hmeet [b|_*g]\hdom \bigg) \bullet \bigg([f{}_*|c]\bigg|_* [f{}_*|c]\hcod \hmeet [g{}_*|d]\hdom \bigg)
\end{align*}
The argument justifying the second and third equalities in this proof is exactly analogous to that in the proof of Proposition \ref{fact2}.
\end{proof}
The following proposition follows almost immediately from the axioms of a double inductive groupoid. These identities will be used to re-express certain arrows by which we will want to (co)restrict.
\begin{proposition} \label{fact3}
\begin{enumerate}[(i)] 
\item $(a|_*u) \vcod \vmeet (c|_*v)\vdom = (a \vcod \vmeet c\vdom |_* u\vcod \vmeet v\vdom)$
\item $(u{}_*|b)\vcod \vmeet (v{}_*|d)\vdom = (u\vcod \vmeet v\vdom {}_*| b\vcod \vmeet d\vdom) $
\item $[a|_*f]\hcod \hmeet [b|_g]\hdom = [a\hcod \hmeet b\hdom|_* f\hcod \hmeet g\hdom]$
\item $[f{}_*|c]\hcod \hmeet [g{}_*|d]\hdom = [f\hcod \hmeet g\hdom {}_*| c\hcod \hmeet d\hdom]$
\end{enumerate}
\end{proposition}
\begin{proof}
We will prove only (i) -- the proofs of (ii), (iii) and (iv) are analogous. We apply first the functoriality of vertical domains and codomains with respect to horizontal (co)restrictions followed by the interchange law for vertical meets and horizontal (co)restrictions.
\begin{align*}
(a|_*u) \vcod \vmeet (c|_*v)\vdom &= (a\vcod |_*u\vcod) \vmeet (c\vdom|_*v\vdom) \\
&= (a \vcod \vmeet c\vdom |_* u\vcod \vmeet v\vdom)
\end{align*}
\end{proof}
\noindent{\bf Proof of Proposition \ref{m4dig}} \\
Consider the semigroup product
\[ (a\circledcirc b)\odot (c\circledcirc d) = \big[a\circledcirc b \big|_* (a\circledcirc b)\vcod \vmeet (c\circledcirc d)\vdom\big] \bullet \big[(a\circledcirc b)\vcod \vmeet (c\circledcirc d)\vdom {}_*\big| c\circledcirc d]. \]
Note that 
\[ a\circledcirc b = (a |_* a\hcod \hmeet b\hdom) \circ (a\hcod \hmeet b\hdom {}_*| b)\]
and 
\[ c\circledcirc d = (c |_* c\hcod \hmeet d\hdom) \circ (c\hcod \hmeet d\hdom {}_*| d)\]
so that by Proposition \ref{fact1} with $u = a\hcod \hmeet b\hdom$ and $v = c\hcod \hmeet d\hdom,$
\begin{align*}
&(a\circledcirc b)\odot (c\circledcirc d) \\
&\hspace{1pc}= \bigg\{\bigg[ (a|_*u)\bigg|_* (a|_*u)\vcod \vmeet (c|_*v)\vdom \bigg] \circ \bigg[(u{}_*|b)\bigg|_* (u{}_*|b)\vcod \vmeet (v{}_*|d)\vdom \bigg]\bigg\}\\
&\hspace{3pc} \bullet \bigg\{ \bigg[ (a|_*u)\vcod \vmeet (c|_*v)\vdom {}_*\bigg| (c|_*v) \bigg]\circ \bigg[ (u{}_*|b)\vcod \vmeet (v{}_*|d)\vdom {}_*\bigg| (v{}_*|d) \bigg]\bigg\}.
\end{align*}
Using the interchange law of double cells of a double inductive groupoid, we have
{\footnotesize \begin{align*}
&(a\circledcirc b)\odot (c\circledcirc d) \\
&{ \hspace{1pc}=\bigg\{\bigg[ (a|_*u)\bigg|_* (a|_*u)\vcod \vmeet (c|_*v)\vdom \bigg] \bullet\bigg[ (a|_*u)\vcod \vmeet (c|_*v)\vdom {}_*\bigg| (c|_*v) \bigg] \bigg\}}\\
&\hspace{3pc} \circ \bigg\{\bigg[(u{}_*|b)\bigg|_* (u{}_*|b)\vcod \vmeet (v{}_*|d)\vdom \bigg] \bullet  \bigg[ (u{}_*|b)\vcod \vmeet (v{}_*|d)\vdom {}_*\bigg| (v{}_*|d) \bigg]\bigg\}.
\end{align*}}
Applying Propositions \ref{fact3}(i) and \ref{fact3}(ii), we have
{\footnotesize \begin{align*} 
&(a\circledcirc b)\odot (c\circledcirc d)\\
&{ = \bigg\{\bigg[ (a|_*u)\bigg|_* (a \vcod \vmeet c\vdom |_* u\vcod \vmeet v\vdom) \bigg] \bullet\bigg[ (a \vcod \vmeet c\vdom |_* u\vcod \vmeet v\vdom) {}_*\bigg| (c|_*v) \bigg] \bigg\}}\\
&{\hspace{1pc} \circ \bigg\{\bigg[(u{}_*|b)\bigg|_* (u\vcod \vmeet v\vdom {}_*| b\vcod \vmeet d\vdom) \bigg] \bullet  \bigg[ (u\vcod \vmeet v\vdom {}_*| b\vcod \vmeet d\vdom) {}_*\bigg| (v{}_*|d) \bigg]\bigg\}. }
\end{align*}}
Using Property (\ref{restcomm}) from Definition \ref{doubleinductivegroupoids} of double inductive groupoids, we have
{\footnotesize \begin{align*}
&(a\circledcirc b)\odot (c\circledcirc d)\\
&{\hspace{1pc}=\bigg\{\bigg( [a|_*a \vcod \vmeet c\vdom]\bigg|_* [u |_* u\vcod \vmeet v\vdom] \bigg) \bullet\bigg( [a \vcod \vmeet c\vdom {}_*| c] \bigg|_* [u\vcod \vmeet v\vdom {}_*| v] \bigg) \bigg\}}\\
&{\hspace{1pc} \circ \bigg\{\bigg([u|_*u\vcod \vmeet v\vdom]{}_*\bigg| [b |_* b\vcod \vmeet d\vdom] \bigg) \bullet  \bigg( [u\vcod \vmeet v\vdom {}_*| v] {}_*\bigg| [b\vcod \vmeet d\vdom{}_*|d] \bigg)\bigg\}.}
\end{align*}}
Note that, by functoriality of (co)domains, commutativity of (co)domains and interchange law of vertical and horizontal meets, we have
{\small \begin{align*}
[u|_* u\vcod \vmeet v\vdom] 
&= [u|_* (a\hcod \hmeet b\hdom)\vcod \vmeet (c\hcod \hmeet d\hdom)\vdom] \\
&= [u|_* (a\vcod\hcod \hmeet b\vcod\hdom) \vmeet (c\vdom\hcod \hmeet d\vdom\hdom)]\\
&= [u|_* (a\vcod\hcod \vmeet c\vdom\hcod ) \hmeet (b\vcod\hdom \vmeet d\vdom\hdom)]\\
&= [a\hcod \hmeet b\hdom|_* (a\vcod \vmeet c\vdom)\hcod  \hmeet (b\vcod \vmeet d\vdom)\hdom].
\end{align*}}
We leave the similar calculations involving the other three (co)restrictees to the reader. Let $f = a\vcod \vmeet c\vdom$ and $g = b\vcod \vmeet d\vdom.$ Then by Propositions \ref{fact3}(iii) and \ref{fact3}(iv), we have
\begin{align*}
&(a\circledcirc b)\odot (c\circledcirc d)\\
&\hspace{1pc}=\bigg\{ \bigg( [a|_*f]\bigg|_* [a|_*f]\hcod \hmeet [b|_*g]\hdom \bigg) \bullet \bigg([f{}_*|c]\bigg|_* [f{}_*|c]\hcod \hmeet [g{}_*|d]\hdom \bigg)\bigg\} \\
&\hspace{3pc} \circ \bigg\{\bigg([a|_*f]\hcod \hmeet [b|_*g]\hdom {}_*\bigg| [b|_*g]\bigg) \bullet \bigg( [f{}_*|c]\hcod \hmeet [g{}_*|d]\hdom {}_*\bigg| [g{}_*|d]\bigg)\bigg\}.
\end{align*}
By Proposition \ref{fact2},
\begin{align*}
(a\circledcirc b)\odot (c\circledcirc d)
&= \bigg([a|_*f]\bullet [f{}_*|c] \bigg|_* \big\{ [a|_*f]\bullet [f{}_*|c]\big\}\hcod \hmeet \big\{[b|_*g]\bullet [g{}_*|d]\big\}\hdom \bigg) \\
&\hspace{3pc} \circ \bigg(\big\{ [a|_*f]\bullet [f{}_*|c]\big\}\hcod \hmeet \big\{[b|_*g]\bullet [g{}_*|d]\big\}\hdom{}_*\bigg| [b|_*g]\bullet [g{}_*|d]\bigg).
\end{align*}
Recall that 
\[ a\odot c = [a|_* a\vcod \vmeet c\vdom] \bullet [a\vcod \vmeet c\vdom {}_* c]\]
and
\[ b\odot d = [b |_* b\vcod \vmeet c\vdom] \bullet [b\vcod \vmeet c\vdom {}_*| d].\]
By our choice of $f$ and $g,$ then,
\begin{align*}
&(a\circledcirc b)\odot (c\circledcirc d)\\
&\hspace{1pc}= \bigg(a\odot c \bigg|_* \big\{ a\odot c\big\}\hcod \hmeet \big\{b\odot d \big\}\hdom \bigg)   \circ \bigg(\big\{a\odot c\big\}\hcod \hmeet \big\{b\odot d\big\}\hdom{}_*\bigg| b\odot g\bigg) \\
&\hspace{1pc}= (a\odot c) \circledcirc (b\odot d).
\end{align*}
\end{appendices}

\bibliographystyle{plain}
\bibliography{/Users/darien/Dropbox/reference}

\end{document}